\documentclass{article}

\usepackage{arxivMod}

\usepackage[utf8]{inputenc} 
\usepackage[T1]{fontenc}    
\usepackage{hyperref}       
\usepackage{url}            
\usepackage{booktabs}       
\usepackage{amsfonts}       
\usepackage{nicefrac}       
\usepackage{microtype}      
\usepackage{lipsum}		
\usepackage{graphicx}
\usepackage{natbib}
\usepackage{doi}

\RequirePackage{amsthm,amsmath,amsfonts,amssymb}

\RequirePackage{bm}
\usepackage{xcolor}

\theoremstyle{plain}

\newtheorem{theorem}{Theorem}
\newtheorem{corollary}{Corollary}
\newtheorem{lemma}{Lemma}

\newcommand{\lo}{\overline{\bm{\lambda}\vphantom{\bm{\alpha}}}}
\newcommand{\ld}{\overline{\bm{\lambda}^2}}
\newcommand{\lt}{\overline{\bm{\lambda}^3}}

\newcommand{\mo}{\overline{\bm{\mu}\vphantom{\bm{\lambda}}}}
\newcommand{\md}{\overline{\bm{\mu}^2}}
\newcommand{\mt}{\overline{\bm{\mu}^3}}

\newcommand{\ldc}{\overline{\bm{\lambda}_c^2}}
\newcommand{\ltc}{\overline{\bm{\lambda}_c^3}}
\newcommand{\lqc}{\overline{\bm{\lambda}_c^4}}

\newcommand{\mdc}{\overline{\bm{\mu}_c^2}}
\newcommand{\mtc}{\overline{\bm{\mu}_c^3}}
\newcommand{\mqc}{\overline{\bm{\mu}_c^4}}

\newcommand{\bbGamma}{{\mathpalette\makebbGamma\relax}}
\newcommand{\makebbGamma}[2]{%
  \raisebox{\depth}{\scalebox{1}[-1]{$\mathsurround=0pt#1\mathbb{L}$}}}

\title{Exact first moments of the RV coefficient by invariant orthogonal integration}

\author{Fran\c{c}ois Bavaud\\
	University of Lausanne, Switzerland\\
		\texttt{fbavaud@unil.ch}  \\
}

\renewcommand{\shorttitle}{Exact first moments of the RV coefficient by invariant orthogonal integration}

\hypersetup{
pdftitle={Exact first moments of the RV coefficient by invariant orthogonal integration},
pdfsubject={q-bio.NC, q-bio.QM},
pdfauthor={Fran\c{c}ois Bavaud},
pdfkeywords={First keyword, Second keyword, More},
}

\begin{document}
\maketitle

\begin{abstract}
The RV coefficient measures the similarity between two multivariate configurations, and its significance testing has attracted various proposals in the last decades.  We present a new approach, the invariant orthogonal integration, permitting to obtain the exact first four moments of the RV coefficient under the null hypothesis. It consists in averaging along the Haar measure the respective orientations of the two configurations, and can be applied to any multivariate setting endowed with Euclidean distances between the observations. Our proposal also covers the weighted setting of observations of unequal importance, where the exchangeability assumption, justifying the usual permutation tests, breaks down.

The proposed RV moments express as simple functions of the kernel eigenvalues occurring in the weighted multidimensional scaling of the two configurations.  The expressions for the third and fourth moments seem original. The first three moments can be obtained by elementary means, but  computing the fourth moment requires a more sophisticated apparatus, the Weingarten calculus for orthogonal groups. The central role of standard kernels  and their spectral moments is emphasized. 
\end{abstract}


\section{Introduction}
The RV coefficient is a well-known  measure of similarity between two datasets, each consisting of multivariate profiles measured on the same $n$ observations or objects. This contribution proposes a new approach, the {\em  invariant orthogonal integration}, permitting to obtain the exact first four moments of the RV coefficient under the null hypothesis of absence of relation between the two datasets. The main results,  theorem \ref{CVtheo} and corollary \ref{RVcol}, are exposed in 
section \ref{secmarsit}. The approach is fully nonparametric, and allows the handling of {\em weighted objets}, typically made of  {\em aggregates} such as  regions, documents or species, which abound in multivariate analysis. 

In the present distance-based data-analytic approach, data sets are constituted by {\em weighted configurations} specified by the object weights together with their pair dissimilarities, assumed to be squared Euclidean. Factorial coordinates, reproducing the dissimilarities, and permitting a maximum compression of the configuration inertia, obtain by {\em weighted multidimensional scaling}. The latter, seldom exposed in the literature and hence briefly recalled in section \ref{secWmssk}, is a direct generalization of classical scaling. The central step is provided by the spectral decomposition of the matrix of  weighted centered scalar products or {\em kernel}. It permits to decompose the spectral eigenspace into a trivial one-dimensional part, determined by the object weights, common to both configurations, and a non-trivial  part of  dimension $n-1$, orthogonal to the square root of the weights. The weighted RV coefficient obtains as the normalized scalar product between the kernels of the two configurations (section \ref{presRVco}), and turns out to be equivalent to its original definition expressed by cross-covariances \citep{escoufier1973traitement, robert1976unifying}. 

After recalling the above preliminaries, somewhat lengthy but necessary,  the heart of this contribution can be uncovered: invariant orthogonal integration consists in computing the expected null moments of the RV coefficient by averaging,   along the invariant Haar orthogonal measure in the non-trivial eigenspace, 
the orientations of one configuration  with respect to the other, by orthogonal transformation of, say, the first eigenspace (section \ref{invorthin}). It constitutes a distinct alternative, with different outcomes, to the traditional permutation approach, whose exchangeability assumption breaks down for weighted objects: typically, the profile dispersion is expected to be larger for lighter objects  \citep{bavaud2013testing} and  the $n$ object scores cannot follow the same distribution. The present approach also 
yields a novel significance test  for the RV coefficient (equation \ref{sigTest}), taking into account skewness and kurtosis corrections to the usual normal approximation. 
 
Computing the moments of the RV coefficient requires to evaluate the  {\em orthogonal coefficients} (\ref{orthoCoeff}) constituted by Haar expectations of orthogonal monomials. Low-order moments can be computed, with increasing difficulty, by elementary means (section \ref{Computing low-order}), but the fourth-order moment requires a more systematic approach (section \ref{q4}), provided by the {\em Weingarten calculus} developed by workers in random matrix theory and free probability.  Both  procedures yield the same results for low-order moments (section \ref{q3rev}), which is both expected and reassuring. 

The first  RV moment (\ref{4MainCoro1}) coincides with all known proposals. The second centered RV moment (\ref{4MainCoro2}) is simpler than its permutation analog, and underlines the {\em effective dimensionality} of a configuration. The third centered RV moment (\ref{4MainCoroo3}) is particularly enlightening: the RV skewness is simply proportional to the product of the {\em spectral skewness} of both configurations, thus elucidating the often noticed  positive skewness of the RV coefficient. The expression for the fourth centered RV moment  (\ref{4MainTheo4}), (\ref{4MainCoro4}) is also simple to express and to compute, yet more difficult to interpret.

\section{Euclidean configurations in a weighted setting: a concise remainder}
\label{Ecwsqr}
\subsection{Weighted multidimensional scaling and standard kernels}
\label{secWmssk}
Consider  $n$ objects endowed with positive {\em weights} $f_i>0$ with $\sum_{i=1}^n f_i=1$, as well with pairwise {\em dissimilarities} $\mathbf{D}=(D_{ij})$  between pairs of objects. The $n\times n$ matrix $\mathbf{D}$ is assumed to be 
 {\em squared Euclidean}, that is of the  form $D_{ij}=\|\mathbf{x}_i-\mathbf{x}_j\|^2$ for $\mathbf{x}_i,\mathbf{x}_j\in \mathbb{R}^r$, with $r\le n-1$. The pair $(\mathbf{f},\mathbf{D})$ constitutes a  {\em weighted configuration}, with $f_i=1/n$ for unweighted configurations.

Weighted multidimensional scaling aims at determining object coordinates $\mathbf{X}=(x_{i\alpha})\in \mathbb{R}^{n\times r}$  reproducing the dissimilarities $\mathbf{D}$ while expressing a maximum amount of dispersion or {\em inertia} $\Delta$ (\ref{inertia}) in low dimensions. It is performed by the following weighted generalization of the well-known Torgerson–Gower scaling procedure \citep[see e.g.][]{borg2005modern}: first, define $\bm{\Pi}=\mbox{diag}(\mathbf{f})$, as well as the weighted centering matrix $\mathbf{H}=\mathbf{I}_n-\bm{1}_n\mathbf{f}^\top$, which obeys $\mathbf{H}^2=\mathbf{H}$. However, $\mathbf{H}^\top\neq \mathbf{H}$, unless $\mathbf{f}$ is uniform. 

Second, compute the  matrix $\mathbf{B}$ of {\em scalar products} by double centering:  $\mathbf{B}=-\frac12 \mathbf{H}\, \mathbf{D}\, \mathbf{H}^\top$. Third, define the 
$n\times n$ {\em kernel} $\mathbf{K}$ as the matrix of {\em weighted scalar products} : 
\begin{displaymath}
\mathbf{K}=\sqrt{\bm{\Pi}}\,\mathbf{B}\sqrt{\bm{\Pi}}\qquad,\qquad\mbox{that is}\qquad K_{ij}=\sqrt{f_if_j}B_{ij}\enspace. 
\end{displaymath}
Fourth,  perform the spectral decomposition with $\hat{\mathbf{U}}$ orthogonal and $\hat{\bm{\Lambda}}$ diagonal
\begin{equation}
\label{specdec}
\mathbf{K}=\hat{\mathbf{U}}\hat{\bm{\Lambda}}\hat{\mathbf{U}}^\top \qquad\qquad \hat{\mathbf{U}}\hat{\mathbf{U}}^\top=\hat{\mathbf{U}}^\top \hat{\mathbf{U}}=\mathbf{I}_n
\qquad\qquad \hat{\bm{\Lambda}}=\mbox{diag}(\bm{\lambda})\enspace. 
\end{equation}
By construction, $\mathbf{K}$  possesses one trivial eigenvalue $\lambda_0=0$ associated to the eigenvector $\sqrt{\mathbf{f}}$ and $n-1$ non-negative eigenvalues decreasingly ordered as $\lambda_1\, \ge\, \lambda_2\, \ge\,\ldots\,\ge\, \lambda_{n-1}\ge0$, among which $r=\mbox{rg}(\mathbf{K})$ are strictly positive.

From now on {\em the trivial eigenspace will be  discarded}: set $\hat{\mathbf{U}}=(\sqrt{\mathbf{f}}|\mathbf{U})$, where $\mathbf{U}\in \mathbb{R}^{n\times (n-1)}$ and $\bm{\Lambda}=\mbox{diag}(\lambda_1,\ldots,\lambda_{n-1})$. Direct substitution from (\ref{specdec}) yields
\begin{equation}
\label{Knontrivial}
\mathbf{K}= \mathbf{U} \bm{\Lambda} \mathbf{U}^\top \quad\qquad\mathbf{U}\mathbf{U}^\top=\mathbf{I}_n-\sqrt{\mathbf{f}}\sqrt{\mathbf{f}}^\top\qquad\quad
\mathbf{U}^\top\mathbf{U}=\mathbf{I}_{n-1}\qquad\quad \mathbf{U}^\top\sqrt{\mathbf{f}}=\mathbf{0}_n\enspace. 
\end{equation}

Finally, the searched for coordinates obtain as $\mathbf{X}=\bm{\Pi}^{-\frac12} \mathbf{U} \bm{\Lambda}^\frac12$, that is $x_{i\alpha}=u_{i\alpha}\sqrt{\lambda_\alpha}/\sqrt{f_i}$. One verifies easily that
\begin{equation}
\label{inertia}
D_{ij}=\sum_{\alpha=1}^{n-1} (x_{i\alpha}-x_{j\alpha})^2\qquad\qquad\qquad \Delta=\frac12\sum_{i,j=1}^n f_i f_j D_{ij}=\mbox{Tr}(\mathbf{K})=\sum_{\alpha=1}^{n-1} \lambda_\alpha\enspace. 
\end{equation}
The kernels considered here are positive semi-definite and obey in addition $\mathbf{K}\sqrt{\mathbf{f}}=\mathbf{0}_n$. We call them {\em standard kernels}. They can be related to the weighted version of centered kernels of Machine Learning \citep[see e.g.][]{cortes2012algorithms}. To each weighted configuration $(\mathbf{f},\mathbf{D})$ corresponds a unique standard kernel $\mathbf{K}$, and conversely. 

The matrix $\mathbf{K}_0=\mathbf{I}_n-\sqrt{\mathbf{f}}\sqrt{\mathbf{f}}^\top$ appearing in (\ref{Knontrivial}) constitutes a standard kernel, referred to as the {\em neutral kernel} in view of property $\mathbf{K}_0\mathbf{K}=  \mathbf{K}_0\mathbf{K}=\mathbf{K}$ for any standard kernel $\mathbf{K}$. 
The corresponding dissimilarities  are the {\em  weighted discrete distances} 
\begin{displaymath}
D_{ij}^0= \begin{cases}
  \frac{1}{f_i}+\frac{1}{f_j}  &   \text{ for $i\neq j$} \\
 0 &  \text{ otherwise.}
\end{cases}
\end{displaymath}

\subsection{The RV coefficient}
\label{presRVco}
Consider  two weighted configurations $(\mathbf{f},\mathbf{D}_X)$ and  $(\mathbf{f},\mathbf{D}_Y)$ endowed with the {\em same weights} $\mathbf{f}$, or equivalently two standard kernels $\mathbf{K}_X$ and $\mathbf{K}_Y$ (Figure~\ref{two_configs}). Their similarity can be measured by the {\em weighted RV coefficient} defined as
\begin{equation}
\label{RVcoeffi}
{\tt RV}={\tt RV}_{XY}=\frac{\mbox{Tr}(\mathbf{K}_X\; \mathbf{K}_Y)}{\sqrt{\mbox{Tr}(\mathbf{K}^2_X)\mbox{Tr}(\mathbf{K}^2_Y)}}
\end{equation}
which constitutes the cosine similarity between the vectorized matrices $\mathbf{K}_X$ and $\mathbf{K}_Y$. As a consequence, ${\tt RV}_{XY}\ge0$ (since $\mathbf{K}_X$  and $\mathbf{K}_Y$ are positive semi-definite), ${\tt RV}_{XY}\le1$ (by the Cauchy-Schwarz inequality) and ${\tt RV}_{XX}=1$.

\begin{figure}
\begin{center}
\includegraphics[width=0.3\textwidth]{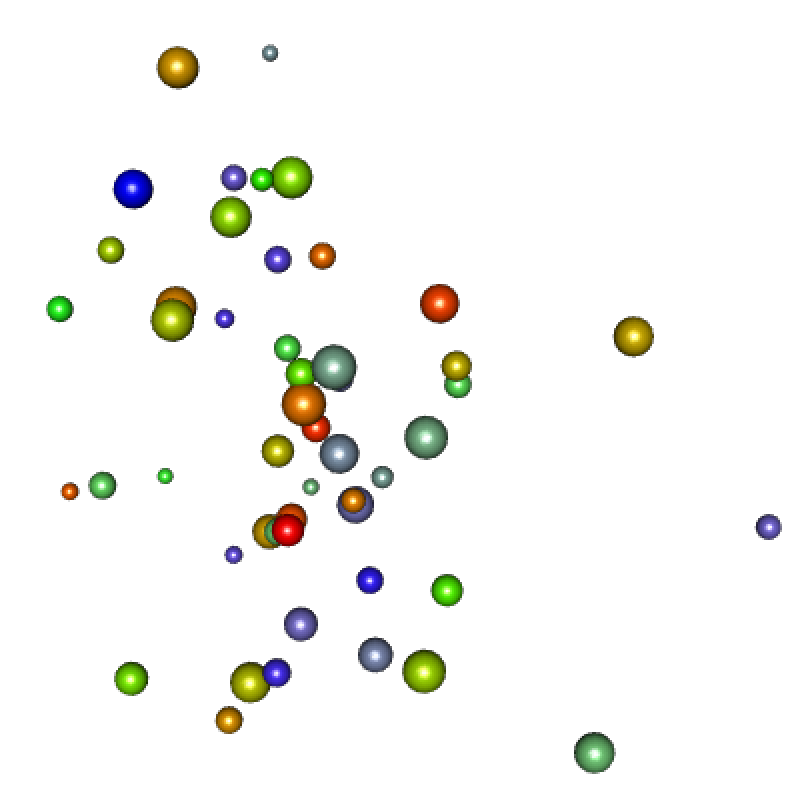}  
\hspace{2.5cm}
\includegraphics[width=0.26\textwidth]{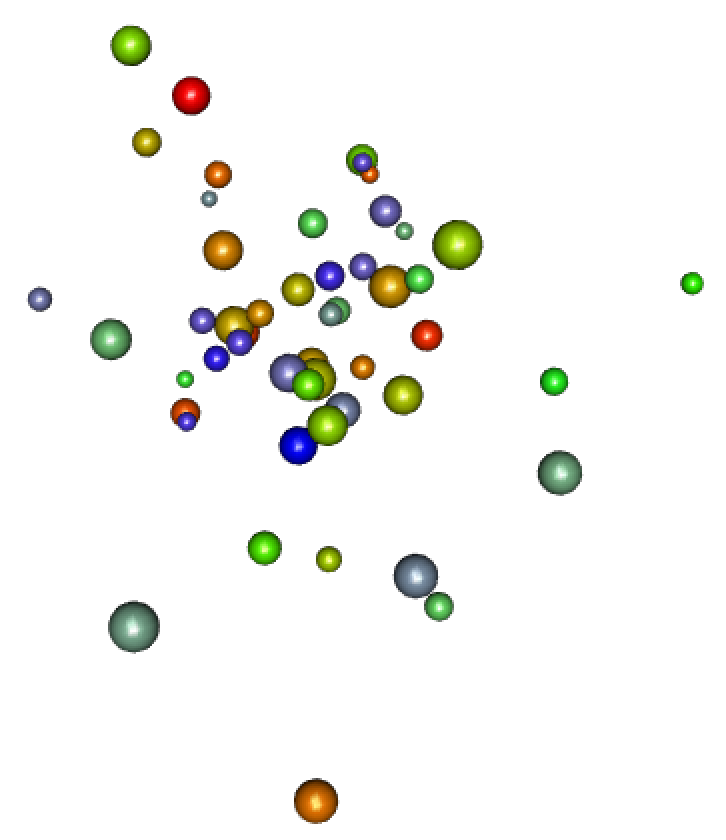}
\caption{Two weighted configurations  $(\mathbf{f},\mathbf{D}_X)$ (left) and  $(\mathbf{f},\mathbf{D}_Y)$ (right) embedded in $\mathbb{R}^{n-1}$}
\label{two_configs}
\end{center}
\end{figure}

Quantity (\ref{RVcoeffi}) is a straightforward weighted generalization of the RV coefficient  \citep{escoufier1973traitement, robert1976unifying}: 
consider multivariate features $\mathbf{X}\in \mathbb{R}^{n\times p}$ and $\mathbf{Y}\in \mathbb{R}^{n\times q}$, directly entering into the 
definition of $\mathbf{D}_X$ and $\mathbf{D}_Y$ as coordinates, or equivalently as   $\mathbf{K}_X=\sqrt{\bm{\Pi}}\mathbf{X}_c\mathbf{X}_c^\top\sqrt{\bm{\Pi}}$ and $\mathbf{K}_Y=\sqrt{\bm{\Pi}}\mathbf{Y}_c\mathbf{Y}_c^\top\sqrt{\bm{\Pi}}$, where $\mathbf{X}_c=\mathbf{H}\mathbf{X}$ and $\mathbf{Y}_c=\mathbf{H}\mathbf{Y}$ are the centered scores.

The weighted covariances are
$\bm{\Sigma}_{XX}=\mathbf{X}_c^\top \bm{\Pi}\mathbf{X}_c$ and $\bm{\Sigma}_{YY}=\mathbf{Y}_c^\top \bm{\Pi}\mathbf{Y}_c$. The {\em cross-covariances} are  $\bm{\Sigma}_{XY}=\mathbf{X}_c^\top \bm{\Pi}\mathbf{Y}_c$ and $\bm{\Sigma}_{YX}=\mathbf{Y}_c^\top \bm{\Pi}\mathbf{X}_c=\bm{\Sigma}_{XY}^\top$. The  original RV coefficient is defined in the feature space as 
\begin{equation}
\label{RVcoeffiCOV}
{\tt RV}_{XY}=\frac{\mbox{Tr}(\bm{\Sigma}_{XY}\; \bm{\Sigma}_{YX})}{\sqrt{\mbox{Tr}(\bm{\Sigma}^2_{XX})\mbox{Tr}(\bm{\Sigma}^2_{YY})}}\enspace. 
\end{equation}
Proving  the identity of  (\ref{RVcoeffi}) and (\ref{RVcoeffiCOV}) is easy.

\section{Computing the moments of the RV coefficient by invariant orthogonal integration}
\label{semoninsk}

\subsection{Main result and significance testing} 
\label{secmarsit}

$\mbox{}$ 

Define the CV coefficient by the quantity ${\tt CV}=\mbox{Tr}(\mathbf{K}_X \mathbf{K}_Y)$.

\label{mainres}
\begin{theorem}[Main result]
\label{CVtheo}
Under invariant orthogonal integration (section \ref{invorthin}), the expectation of the CV coefficient and its centered moments of order 2, 3 and 4 are
\begin{align}
\label{4MainTheo1}
 \mathbb{E}({\tt CV})=  & (n-1)\, \lo\, \mo \\
\label{4MainTheo2}
\mathbb{E}({\tt CV}_c^2)=  &   \frac{2(n-1)^2}{(n-2)(n+1)}\, \ldc\;  \mdc \\
\label{4MainTheo3}
\mathbb{E}({\tt CV}_c^3)=  &  \frac{8(n-1)^3}{(n-3)(n-2)(n+1)(n+3)} \, \ltc\;  \mtc \\
\label{4MainTheo4}
\mathbb{E}({\tt CV}_c^4)=  &  {  \frac{12\, (n-1)^3}{(n-4)(n-3)(n-2)n(n+1)(n+3)(n+5)}  }\, \Bigl\{  4 \,  (n^2-n+2)\, \lqc\: \mqc     +(n^4 \\ 
 +n^3  -15 n^2& -13 n+98)\, \ldc\: \ldc\:  \mdc\:  \mdc   -   4\, (2 n^2-n-7)\, (\lqc\:  \mdc\: \mdc+\ldc\: \ldc\:  \mqc)\Bigr\}\notag\enspace . 
\end{align}
\end{theorem}
where ${\tt CV}_c={\tt CV}-\mathbb{E}({\tt CV})$. Spectral moments and  centered spectral moments read 
\begin{equation}
\label{spmocespmo}
\overline{\bm{\lambda}^q}=\frac{1}{n-1}\sum_{\alpha=1}^n\lambda_\alpha^q=\frac{1}{n-1}\sum_{\alpha=1}^n\mbox{Tr}(\mathbf{K}_X^q)=\mbox{tr}(\mathbf{K}_X^q)
\qquad\qquad
\overline{\bm{\lambda}^q_c}=\frac{1}{n-1}\sum_{\alpha=1}^n(\lambda^c_\alpha)^q 
\end{equation}
where  $\lambda^c_\alpha=\lambda_\alpha-\overline{\bm{\lambda}}$ and  $\mbox{tr}(\mathbf{A})= \mbox{Tr}(\mathbf{A})/(n-1)$ denotes the {\em normalized trace}. Centered spectral moments can be transformed into normalized traces, and conversely. For instance, $\ltc=\mbox{tr}(\mathbf{K}_X^3)-3\, \mbox{tr}(\mathbf{K}_X^2)\, \mbox{tr}(\mathbf{K}_X)+2\, \mbox{tr}^3(\mathbf{K}_X)$.

Identity ${\tt RV}={\tt CV}/\sqrt{\mbox{Tr}(\mathbf{K}^2_X)\mbox{Tr}(\mathbf{K}^2_Y)}= {\tt CV}/[(n-1) \sqrt{\ld\;  \md}\, ]$ directly yields:  

\begin{corollary}[First cumulants of the RV coefficient]
\label{RVcol}
Under invariant orthogonal integration, the first cumulants of the RV coefficient, that is its expectation, variance, skewness and excess kurtosis are, in order, 
\begin{align}
 \label{4MainCoro1}
\mathbb{E}({\tt RV})=  & \frac{1}{n-1}\frac{\mbox{\rm Tr}(\mathbf{K}_X)\, \mbox{\rm Tr}(\mathbf{K}_Y)}{\sqrt{\mbox{\rm Tr}(\mathbf{K}_X^2)\, \mbox{\rm Tr}(\mathbf{K}_Y^2)}}
 =\frac{\overline{\bm{\lambda}}\overline{\bm{\mu}\: \vphantom{\lambda}}}{\sqrt{\overline{\bm{\lambda}^2}\: \overline{\bm{\mu}^2}}}=\frac{\sqrt{\nu({\bm \lambda})\nu({\bm \mu})}}{n-1}
 \\
\label{4MainCoro2}
\mbox{$\mathbb{V}$\rm ar}({\tt RV}) = & \mathbb{E}({\tt RV}_c^2)=     \frac{2(n-1-\nu({\bm \lambda}))(n-1-\nu({\bm \mu}))}{(n-2)(n-1)^2(n+1)} \\
\label{4MainCoroo3}
 \mbox{$\mathbb{A}$}({\tt RV})= & \frac{\mathbb{E}({\tt RV}_c^3)}{\mathbb{E}^{\frac32}({\tt RV}_c^2)}=   \frac{\sqrt{8(n-2)(n+1)}}{(n-3)(n+3)}\:   a({\bm \lambda})\:  a({\bm \mu}) \\
\label{4MainCoro4}
 \bbGamma({\tt RV})= & \frac{\mathbb{E}({\tt RV}_c^4)}{\mathbb{E}^2({\tt RV}^2_c)}-3 = \frac{\mathbb{E}({\tt CV}_c^4)}{\mathbb{E}^2({\tt CV}_c^2)}-3
 \qquad\quad\mbox{\rm (see (\ref{4MainTheo4}) and (\ref{4MainTheo2}))}
\end{align}
\end{corollary}
where ${\tt RV}_c={\tt RV}-\mathbb{E}({\tt RV})$, and $a({\bm \lambda})=\ltc/(\ldc)^\frac32$ is the {\em spectral skewness}. The quantity 
 \begin{equation}
\label{nux}
\nu({\bm \lambda})=\frac{\mbox{Tr}^2(\mathbf{K}_X)}{\mbox{Tr}(\mathbf{K}_X^2)}=\frac{(\sum_{\alpha\ge1}\lambda_\alpha)^2}{\sum_{\alpha\ge1}\lambda^2_\alpha}=(n-1)\frac{\overline{\bm{\lambda}}^2}{\overline{\bm{\lambda}^2}}
\end{equation}
has appeared at times as an adjusted degrees of freedom  in multivariate tests of the general linear model  
\citep[see e.g.][]{geisser1958extension,worsley1995analysis,schlich1996defining, abdi2010congruence}. It provides a measure of  sphericity or {\em effective dimensionality} of configuration $(\mathbf{f},\mathbf{D}_X)$. Its minimum  $\nu({\bm \lambda})=1$ is attained for univariate configurations. Its maximum  $\nu({\bm \lambda})=n-1$ is attained for uniform dilatations of the discrete distances $\mathbf{D}^0_X$ (section \ref{secWmssk}), in which case $\mbox{$\mathbb{V}$\rm ar}({\tt RV}) =0$ since ${\tt RV}$ is then concentrated on $\sqrt{\nu({\bm \mu})/(n-1)}$.

The second-order Cornish-Fisher cumulant expansion permits
 to approximatively redress the normal quantiles by taking into account the skewness and the "taildeness" of 
a non-normal distribution  \citep[see e.g.][]{kendall1977advanced,amedee2019computation}. The observed RV is statistically significant at level $\alpha$ if (one-tailed test)
\begin{equation}
\label{sigTest}
\begin{split}
&\qquad\qquad\qquad \underbrace{\frac{{\tt RV}- \mathbb{E}({\tt RV})}{\sqrt{\mbox{$\mathbb{V}$ar}({\tt RV})}}}_{\mbox{\small\em $z$-score}} \qquad
\: \: > \: \:   \underbrace{u_{1-\alpha}}_{\mbox{\small\em standard normal quantile}} \qquad \\
     & +\quad \underbrace{\frac{\mbox{$\mathbb{A}$}({\tt RV})}{6}(u^2_{1-\alpha}-1)+
\frac{\bbGamma({\tt RV})}{24}(u^3_{1-\alpha}-3u_{1-\alpha})-\frac{\mbox{$\mathbb{A}^2$}({\tt RV})}{36}(2u^3_{1-\alpha}-5u_{1-\alpha})}_{\mbox{\small\em correction to the normal distribution}}\enspace. 
\end{split}
\end{equation}

\subsection{Invariant orthogonal integration}
\label{invorthin}
The rest of the paper is devoted to presenting invariant orthogonal integration and proving Theorem \ref{CVtheo}. 

Consider two standard kernels $\mathbf{K}_X=\mathbf{U}\bm{\Lambda}\mathbf{U}^\top$ and 
$\mathbf{K}_Y=\mathbf{V}\mathbf{M}\mathbf{V}^\top$ with $\mathbf{M}=\mbox{diag}(\bm{\mu})$, where $\mathbf{U}, \mathbf{V}\in\mathbb{R}^{n\times (n-1)}$  (\ref{Knontrivial}). The numerator ${\tt CV}$ of ${\tt RV}$ in (\ref{RVcoeffi}) reads
\begin{equation}
\label{CVdef1}
{\tt CV}=\mbox{Tr}(\mathbf{K}_X\mathbf{K}_Y)=\mbox{Tr}(\mathbf{U}\bm{\Lambda}\mathbf{U}^\top\mathbf{V}\mathbf{M}\mathbf{V}^\top)=
\sum_{\alpha=1}^{n-1}\sum_{\beta=1}^{n-1} \lambda_{\alpha}\mu_{\beta} P_{\alpha\beta}
\end{equation}
where 
\begin{displaymath}
 P_{\alpha\beta}=\sum_{i,j=1}^n u_{i\alpha}u_{j\alpha}v_{i\beta}v_{j\beta}=(\sum_{i=1}^n u_{i\alpha}v_{i\beta})^2\enspace.
\end{displaymath}
Identities   $\mathbf{U}\mathbf{U}^\top=\mathbf{I}_n-\sqrt{\mathbf{f}}\sqrt{\mathbf{f}}^\top$, $\mathbf{V}^\top\sqrt{\mathbf{f}}=\mathbf{0}_n$ and $\mathbf{V}^\top\mathbf{V}=\mathbf{I}_{n-1}$ from (\ref{Knontrivial})  imply the {\em joint orthogonality} property $\mathbf{V}^\top\mathbf{U}\mathbf{U}^\top \mathbf{V}=\mathbf{I}_{n-1}$, yielding
\begin{equation}
\label{DoublySto}
P_{\bullet\beta}=\sum_{\alpha=1}^{n-1}P_{\alpha\beta}=\sum_{i,j=1}^n (\delta_{ij}-\sqrt{f_if_j})\,  v_{i\beta}v_{j\beta}=\sum_{i=1}^n v_{i\beta}^2=1\enspace
\end{equation}
and, similarly, $P_{\alpha\bullet}=1$.  Hence, the matrix $\mathbf{P}=(P_{\alpha\beta})\in \mathbb{R}^{(n-1)\times(n-1)}$ is non-negative,  
and doubly stochastic: it expresses as a mixture of permutations of $\mathcal{S}_{n-1}$  (Birkhoff–von Neumann theorem). In particular, one gets the crude estimate
\begin{displaymath}
\sum_{\alpha\ge1}\lambda_\alpha\mu_{n-\alpha}\: \le\: {\tt CV}\: \le\: 
\sum_{\alpha\ge1}\lambda_\alpha\mu_{\alpha}\enspace. 
\end{displaymath}
The null hypothesis $H_0$ states that the two configurations  $(\mathbf{f},\mathbf{D}_X)$ and  $(\mathbf{f},\mathbf{D}_Y)$ are unrelated. Under $H_0$, any
relative orientation of a configuration with respect to the other is equally likely. Hence, the first configuration will be rotated by replacing $\mathbf{U}=(u_{i\alpha})$ by $\mathbf{U}\mathbf{T}$, where $\mathbf{T}=(t_{a\alpha}) \in \mathbb{O}_{n-1}$, the orthogonal group of dimension $n-1$. This rotation acts in the non-trivial eigenspace only, leaving the weights $\mathbf{f}$ unchanged. The
term $\mbox{Tr}(\mathbf{K}^2_X)$ remains the same, and the ${\tt CV}$ coefficient becomes
\begin{equation}
\label{CVT}
{\tt CV}(\mathbf{T})=\mbox{Tr}(\mathbf{U}\mathbf{T}\bm{\Lambda}\mathbf{T}^\top\mathbf{U}^\top\mathbf{V}\mathbf{M}\mathbf{V}^\top)=
\sum_{\alpha=1}^{n-1}\sum_{\beta=1}^{n-1} \lambda_{\alpha}\mu_{\beta} P_{\alpha\beta}(\mathbf{T})
\end{equation}
where 
\begin{equation}
\label{PabT}
P_{\alpha\beta}(\mathbf{T})=\sum_{a,b=1}^{n-1}t_{a\alpha}t_{b\alpha}\sum_{i,j=1}^n u_{ia}u_{jb}v_{i\beta}v_{j\beta}\enspace. 
\end{equation}
The idea  of {\em invariant orthogonal integration} is to compute the expectation of the moments 
\begin{equation}
\label{momCVq}
\mathbb{E}({\tt CV}^q):=\int_{\mathbb{O}_{n-1}}{\tt CV}^q(\mathbf{T})\: d\mu(\mathbf{T})\qquad\qquad q=1,2,\ldots
\end{equation}
by averaging over all possible rotations $\mathbf{T}\in  \mathbb{O}_{n-1}$ distributed 
by  the  {\em invariant Haar measure} $d\mu(\mathbf{T})$  normalized to $\int_{\mathbb{O}_{n-1}}d\mu(\mathbf{T})=1$.  The moment  generating function reads
\begin{equation}
\label{MomGneFun}
\mathbb{E}(\exp(t\, {\tt CV}))=\int_{\mathbb{O}_{n-1}}\exp(t\: \mbox{Tr}(\mathbf{T}\bm{\Lambda}{\mathbf{T}^\top}\mathbf{A}))\: d\mu(\mathbf{T})
\quad\quad\mbox{with} \quad\mathbf{A}=\mathbf{U}^\top\mathbf{V}\mathbf{M}\mathbf{V}^\top\mathbf{U}\, .
\end{equation}

Define $[n]=\{1,2,\ldots,n\}$. Computing (\ref{momCVq}) involves 
the {\em orthogonal coefficients}, defined in whole generality as 
\begin{equation}
\label{orthoCoeff}
\mathcal{I}_{\mathbf a}^{\bm \omega}=
\int_{\mathbb{O}_{n-1}}d\mu(\mathbf{T})\: t_{a_1 \omega_1}\, t_{a_2 \omega_2}\ldots   t_{a_{2q} \omega_{2q}}
\end{equation}
where  the multi-indices ${\mathbf a}=(a_1a_2\ldots a_{2q})$ and ${\bm  \omega}=( \omega_1 \omega_2\ldots \omega_{2q})$ are elements of $[n-1]^{2q}$
that is, ${\mathbf a}$ and ${\bm  \omega}$ are words of length $2q$ on the alphabet $[n-1]$. 

To ease the notations, define  $A_q=[n-1]^{q}$ and, for ${\bm \alpha}=(\alpha_1\ldots\alpha_q)\in A_q$, define 
 ${\bm \alpha}{\bm \alpha}=(\alpha_1\alpha_1 \alpha_2\alpha_2\ldots\alpha_q\alpha_q)\in  A_{2q}$.  Identities (\ref{CVT}), (\ref{PabT}), (\ref{momCVq}) and (\ref{orthoCoeff}) yield
\begin{align}
\label{heavy}
\mathbb{E}({\tt CV}^q)    & = \sum_{\alpha_1\ldots\alpha_q=1}^{n-1}\underbrace{\lambda_{\alpha_1}\cdots\lambda_{\alpha_q}}_{\lambda_{\bm \alpha}}\sum_{\beta_1\ldots\beta_q=1}^{n-1}  
\underbrace{\mu_{\beta_1}\cdots\mu_{\beta_q}}_{\mu_{\bm \beta}}\sum_{i_1\ldots i_{2q}=1}^{n}\underbrace{v_{i_1\beta_1}v_{i_2\beta_1}v_{i_3\beta_2}v_{i_4\beta_2}\ldots v_{i_{2q-1}\beta_q}v_{i_{2q}\beta_q}}_{v_{{\mathbf i}{\bm \beta}{\bm \beta}}}
\notag  \\
\quad \times    \sum_{a_1\ldots a_{2q}=1}^{n-1}& \underbrace{\mathcal{I}_{a_1a_2\ldots a_{2q-1}a_{2q}}^{\alpha_1\alpha_1\ldots\alpha_q\alpha_q}}_{\mathcal{I}_{\mathbf a}^{{\bm \alpha}{\bm \alpha}}}\underbrace{u_{i_1a_1}u_{i_2a_2}   \ldots u_{i_{2q}a_{2q}}}_{u_{{\mathbf i}{\mathbf a}}}
=\sum_{{\bm \alpha}\in A_q}\lambda_{\bm \alpha}\sum_{{\bm \beta}\in A_q}\mu_{\bm \beta}\underbrace{\sum_{{\mathbf i}\in [n]^{2q}}v_{{\mathbf i}{\bm \beta}{\bm \beta}}
\sum_{{\mathbf a}\in A_{2q}}\mathcal{I}_{\mathbf a}^{{\bm \alpha}{\bm \alpha}}\, u_{{\mathbf i}{\mathbf a}}}_{\mathbb{E}(P_{{\bm \alpha}{\bm \beta}})=\mathbb{E}(P_{\alpha_1\beta_1}P_{\alpha_2\beta_2}\cdots P_{\alpha_q\beta_q})}\, .
\end{align}
The knowledge of $\mathcal{I}_{\mathbf a}^{{\bm \alpha}{\bm \alpha}}$ together with joint orthogonality properties  will yield exact expressions for $\mathbb{E}({\tt CV}^q)$ in terms of   spectral moments of ${\bm \lambda}$ and ${\bm \mu}$, or equivalently in terms of traces of integer powers of $\mathbf{K}_X$ and $\mathbf{K}_Y$, as demonstrated in the next sections for $q=1,2,3,4$.

\subsection{Computing low-order orthogonal coefficients}
\label{Computing low-order}
Evaluating the orthogonal coefficients
(\ref{orthoCoeff}) is a major topic in random matrix theory and free probability, and its systematic handling  is presented in section \ref{q4}. Yet, as observed by some authors
\citep[see e.g.][]{aubert2003invariant, braun2006invariant, yamamoto2017probabilistic}, well-inspired invariance considerations (Lemmas \ref{Mainlemma1} and \ref{Mainlemma2} below) suffice in determining more directly the values of the orthogonal coefficients of low order.
  
Since $d\mu(\mathbf{-T})=d\mu(\mathbf{T})$,  coefficients (\ref{orthoCoeff}) are zero unless each index in ${\mathbf a}$ and in ${\bm\omega}$ occurs 
an even number of times, with a total of $2q$ occurrences, where $q$ defines  the {\em order} of the orthogonal coefficient. Also, applying the same permutation on the two multi-indexes, or exchanging the multi-indexes leaves the coefficients unchanged. Furthermore, the particular value taken by an index is irrelevant: only matters its multiplicity. For instance: 
\begin{equation*}
\label{ }
\mathcal{I}_{abcd}^{\alpha\beta\gamma\delta}=\mathcal{I}_{\alpha\beta\gamma\delta}^{abcd}=\mathcal{I}_{dacb}^{\delta\alpha\gamma\beta}\neq \mathcal{I}_{adcb}^{\delta\alpha\gamma\beta}\quad\mbox{in general ; }
\qquad\qquad \mathcal{I}_{abcc}^{\alpha\alpha\gamma\gamma}=\mathcal{I}_{accb}^{\alpha\alpha\gamma\gamma}=0\quad\mbox{if $a\neq b$}\enspace. 
\end{equation*}
Also, for $\alpha\neq\gamma$ and $a\neq b$, 
\begin{equation*}
\label{ }
\mathcal{I}_{aabb}^{\alpha\alpha\gamma\gamma}=\mathcal{I}_{bbaa}^{\alpha\alpha\gamma\gamma}=\mathcal{I}_{1122}^{2211}=\mathcal{I}_{1122}^{1122}
\enspace. 
\end{equation*}
\begin{lemma}[proved in the Appendix]
\label{Mainlemma1}
Let $\alpha\neq\gamma$ and let ${\bm \varepsilon}$ be a  multi-index not containing $\alpha, \gamma$. For any indices $a,b,c,d$ and multi-index ${\bm e}$ of the same size as ${\bm \varepsilon}$ 
\begin{subequations}
\begin{align}
\mathcal{I}_{abcd{{\bm e}}}^{\alpha\alpha\alpha\alpha{{\bm \varepsilon}}}&=\mathcal{I}_{abcd{{\bm e}}}^{\alpha\alpha\gamma\gamma{{\bm \varepsilon}}}+\mathcal{I}_{abcd{{\bm e}}}^{\alpha\gamma\alpha\gamma{{\bm \varepsilon}}}+ \mathcal{I}_{abcd\mathbf{e}}^{\alpha\gamma\gamma\alpha{{\bm \varepsilon}}} \enspace. 
 \label{lemmeFond}\\
\mbox{In particular,}\qquad  \mathcal{I}_{aacc{\bm e}}^{\alpha\alpha\alpha\alpha{\bm \varepsilon}} &=\mathcal{I}_{aacc{\bm e}}^{\alpha\alpha\gamma\gamma{\bm \varepsilon}}+2\mathcal{I}_{aacc{\bm e}}^{\alpha\gamma\alpha\gamma{\bm \varepsilon}} \label{lemme1}\\
\mbox{and}\qquad   \mathcal{I}_{aaaa{\bm e}}^{\alpha\alpha\alpha\alpha{\bm \varepsilon}} &= 3 \mathcal{I}_{aaaa{\bm e}}^{\alpha\alpha\gamma\gamma\vec{{\bm \varepsilon}}} \enspace. 
\label{lemme2}
\end{align}
\end{subequations}
\end{lemma}

\begin{lemma}
\label{Mainlemma2}
For any multi-index  ${\bm e}$ not containing $a$,  and for any  unrestricted multi-index ${\bm \varepsilon}$ of the same size,
\begin{equation}
\label{lemme3}
 \sum_{a=1}^{n-1}\mathcal{I}^{\alpha\beta{\bm \varepsilon}}_{aa{\bm e}}=\delta_{\alpha\beta}\: \mathcal{I}^{{\bm \varepsilon}}_{\bm e}\enspace. 
 \end{equation}
\end{lemma}

\begin{proof}

(\ref{lemme3}) follows directly from $\bm{T}\bm{T}^\top=\bm{I}_{n-1}$, that is  $\sum_{a=1}^{n-1} t_{a\alpha}t_{a\beta}=\delta_{\alpha\beta}$. 
\end{proof}

\subsection{The first and second moments}
\label{q12}
The computation of the first moment, which has been derived in the literature under various strategies, is straightforward  : $\mathcal{I}_{ab}^{\alpha\alpha}=\delta_{ab}\; \mathcal{I}_{aa}^{\alpha\alpha}$, where 
$\mathcal{I}_{aa}^{\alpha\alpha}$ is independent of $a$ and $\alpha$. By (\ref{CVdef1}), (\ref{heavy}),  (\ref{lemme3}) and joint orthogonality
\begin{equation}
\label{firstorder}
 \mathcal{I}_{aa}^{\alpha\alpha}=\frac{1}{n-1}\qquad\qquad \mathbb{E}(P_{\alpha\beta})=\frac{1}{n-1}\qquad \qquad
\mathbb{E}({\tt CV})=
\frac{1}{n-1}\sum_{\alpha,\beta=1}^{n-1}\lambda_\alpha\mu_\beta\enspace. 
\end{equation}

The computation of the second moment involves four orthogonal coefficients, 
 namely (all super- and sub-indices in (\ref{treize}) are {\em distinct})
\begin{equation}
\label{treize}
E:=\mathcal{I}_{aaaa}^{\alpha\alpha\alpha\alpha}\qquad\quad
F:=\mathcal{I}_{aacc}^{\alpha\alpha\alpha\alpha}=\mathcal{I}_{aaaa}^{\alpha\alpha\gamma\gamma}\qquad\quad
G:=\mathcal{I}_{aacc}^{\alpha\alpha\gamma\gamma}\quad\qquad
H:=\mathcal{I}_{aacc}^{\alpha\gamma\alpha\gamma}\enspace. 
\end{equation}
They satisfy 
\begin{displaymath}
F\stackrel{(\ref{lemme1})}{=}G+2H\qquad
E\stackrel{(\ref{lemme2})}{=}3F\qquad
(n-2)G+F\stackrel{(\ref{lemme3}),(\ref{firstorder})}{=}\frac{1}{n-1}\qquad
(n-2)H+F\stackrel{(\ref{lemme3}),(\ref{firstorder})}{=}0
\end{displaymath}
with solution
\begin{equation}
\label{solkappa}
E=3(n-2)\kappa\: \: ,\: \: 
F=(n-2)\kappa\: \: ,\: \: 
G=n\kappa \: \: ,\: \: 
H=-\kappa\: \: ,\: \: 
\kappa=\frac{1}{(n-2)(n-1)(n+1)}
\end{equation}
Hence (the expression is also valid for four possibly coinciding sub-indices, since $E=3F$)
\begin{equation}
\label{2ordis}
\mathcal{I}_{abcd}^{\alpha\alpha\alpha\alpha}=(\delta_{ab}\delta_{cd}+\delta_{ac}\delta_{bd}+\delta_{ad}\delta_{bc})\,F
\end{equation}
and (for $\alpha\neq\gamma$)
\begin{equation}
\label{2orsim}
\mathcal{I}_{abcd}^{\alpha\alpha\gamma\gamma}=\delta_{ab}\delta_{cd}\,  G+(\delta_{ac}\delta_{bd}+\delta_{ad}\delta_{bc})\,H\enspace. 
\end{equation}
As a result, performing $\delta_{\alpha\gamma}$ $\times$ (\ref{2ordis})  $+$  $(1-\delta_{\alpha\gamma})$ $\times$ (\ref{2orsim}) yields the general  formula, where the super- and the sub-indices may be distinct or not
\begin{equation*}
\label{ }
\mathcal{I}_{abcd}^{\alpha\alpha\gamma\gamma}=\kappa\, [n\delta_{ab}\delta_{cd} -(\delta_{ac}\delta_{bd}+\delta_{ad}\delta_{bc})-2\delta_{\alpha\gamma}\delta_{ab}\delta_{cd} +(n-1)\delta_{\alpha\gamma}(\delta_{ac}\delta_{bd}+\delta_{ad}\delta_{bc})]
\end{equation*}
which finally implies, in view of (\ref{heavy}) and joint orthogonality, 
\begin{equation}
\label{P2}
\mathbb{E}(P_{\alpha\beta}P_{\gamma\delta})=\kappa\, [n-2\delta_{\alpha\gamma}-2\delta_{\beta\delta}+2(n-1)\delta_{\alpha\gamma}\delta_{\beta\delta}]\enspace. 
\end{equation}
Hence, by (\ref{heavy}) and (\ref{spmocespmo})
\begin{equation}
\label{secondorder}
\mathbb{E}({\tt CV}^2)=\kappa(n-1)^3\, [n(n-1)\lo^2\:  \mo^2-2\lo^2\: \md-2\ld \: \mo^2+2\ld\:  \md]
\enspace.
\end{equation}
Substracting $\mathbb{E}^2({\tt CV})$ obtained in  (\ref{firstorder}), substituting the value of $\kappa$ in (\ref{solkappa}) and rearranging terms yields (\ref{4MainTheo2}). 

Expressions for the second centered moments  (\ref{4MainTheo2}) and (\ref{4MainCoro2}) are {\em simpler} than the corresponding quantities obtained, in the unweighted  setting,  by averaging over all permutations of the $n$ objects~: the latter contain additional correction terms, as derived in \citet{kazi1995refined}.
See also \citet{heo1998permutation}, \citet{josse2008testing} and \citet{abdi2010congruence}. 

\subsection{The third moment}
\label{q3}
The third moment reads
\begin{equation}
\label{expres3lm}
\mathbb{E}({\tt CV}^3)=\sum_{\alpha,\beta,\gamma,\delta,\varepsilon,\zeta=1}^{n-1}\lambda_\alpha\, \lambda_\gamma\, \lambda_\varepsilon\, \mu_\beta\, \mu_\delta\, \mu_\zeta
\: \mathbb{E}(P_{\alpha\beta}P_{\gamma\delta}P_{\varepsilon\zeta})
\end{equation}
where
\begin{equation}
\label{CiJ}
{\textstyle\mathbb{E}(P_{\alpha\beta}P_{\gamma\delta}P_{\varepsilon\zeta})  =  \sum_{i,j,k,l,s,t=1}^n  v_{i\beta} v_{j\beta}  v_{k\delta}  v_{l\delta}  v_{s\zeta}  v_{t\zeta}\: \sum_{a,b,c,d,e,f=1}^{n-1} \mathcal{I}_{abcdef}^{\alpha\alpha\gamma\gamma\varepsilon\varepsilon}\: u_{ia}u_{jb}u_{kc} u_{ld}u_{se} u_{tf} }
\end{equation}
and involves eleven third-order orthogonal coefficients, namely (all super- and  sub-indices  in (\ref{seize}) are {\em distinct})
\begin{align}
\label{seize}
 L&:=\mathcal{I}_{aaaaaa}^{\alpha\alpha\alpha\alpha\alpha\alpha}   &   M&:=\mathcal{I}_{aaaacc}^{\alpha\alpha\alpha\alpha\alpha\alpha}   &  N&:=\mathcal{I}_{aaccee}^{\alpha\alpha\alpha\alpha\alpha\alpha} &  P&:=\mathcal{I}_{aaaacc}^{\alpha\alpha\alpha\alpha\gamma\gamma}   \notag\\
Q&:=\mathcal{I}_{aaccaa}^{\alpha\alpha\alpha\alpha\gamma\gamma}   &   R&:=\mathcal{I}_{aaacac}^{\alpha\alpha\alpha\alpha\gamma\gamma}   &  S&:=\mathcal{I}_{aaccee}^{\alpha\alpha\alpha\alpha\gamma\gamma}  &  T&:=\mathcal{I}_{aceeac}^{\alpha\alpha\alpha\alpha\gamma\gamma}  \\
U&:=\mathcal{I}_{aaccee}^{\alpha\alpha\gamma\gamma\varepsilon\varepsilon}   &  V&:=\mathcal{I}_{acacee}^{\alpha\alpha\gamma\gamma\varepsilon\varepsilon}   &  W&:=\mathcal{I}_{aeceac}^{\alpha\alpha\gamma\gamma\varepsilon\varepsilon}   \enspace. &  & \notag 
\end{align}
Handcrafted computations are a bit awkward, yet feasible, with the result

\begin{lemma}[proved in the Appendix]
\label{ThirdOrderP}
\begin{equation}
\label{victoireYes}
\begin{split}
\frac{1}{\hat{\kappa}}\: \mathbb{E}(P_{\alpha\beta}P_{\gamma\delta}P_{\varepsilon\zeta})  & =  (n^2+n-4) -2(n+1)(\sigma+\tau)+16(\varphi+\psi)+8\sigma\tau \\
 & -8(n-1)(\sigma\psi+\tau\varphi) +8(n-1)^2\varphi\psi+2(n-3)(n+3)\omega
\end{split}
\end{equation}
where 
\begin{equation}
\label{kappahat}
\hat{\kappa}=\frac{\kappa}{(n-3)(n+3)}=\frac{1}{(n-3)(n-2)(n-1)(n+1)(n+3)}
\end{equation}
and 
\begin{equation}
\begin{split}
\label{NewDiracs}
 &\sigma =\delta_{\alpha\gamma}+\delta_{\alpha\varepsilon}+\delta_{\gamma\varepsilon}
\qquad  \qquad \quad
\tau=\delta_{\beta\delta}+\delta_{\beta\zeta}+\delta_{\delta\zeta}
\qquad \qquad \quad
\omega=\delta_{\alpha\gamma}\delta_{\beta\delta}+\delta_{\alpha\varepsilon}\delta_{\beta\zeta}+\delta_{\gamma\varepsilon}\delta_{\delta\zeta}
\\
 &\varphi  =\delta_{\alpha\gamma}\delta_{\alpha\varepsilon}\delta_{\gamma\varepsilon}=\delta_{\alpha\gamma}\delta_{\alpha\varepsilon}=
\delta_{\alpha\gamma}\delta_{\gamma\varepsilon}=\delta_{\alpha\varepsilon}\delta_{\gamma\varepsilon}
\qquad
\qquad
\psi=\delta_{\beta\delta}\delta_{\beta\zeta}\delta_{\delta\zeta}=\delta_{\beta\delta}\delta_{\beta\zeta}=\delta_{\beta\delta}\delta_{\delta\zeta}=\delta_{\beta\zeta}\delta_{\delta\zeta}\enspace. 
\end{split}
\end{equation}
\end{lemma}
One can check with (\ref{P2}) that
\begin{equation*}
\label{ }
\sum_{\varepsilon=1}^{n-1}\mathbb{E}(P_{\alpha\beta}P_{\gamma\delta}P_{\varepsilon\zeta})=
\sum_{\zeta=1}^{n-1}\mathbb{E}(P_{\alpha\beta}P_{\gamma\delta}P_{\varepsilon\zeta})
=\mathbb{E}(P_{\alpha\beta}P_{\gamma\delta})
\end{equation*}
as it must. Inserting (\ref{victoireYes}) in ({\ref{expres3lm}) and using  (\ref{spmocespmo}) yields 
\begin{equation}
\label{thirdmoment}
\begin{split}
\frac{\mathbb{E}({\tt CV}^3)}{(n-1)^4\hat{\kappa}}= &
(n^2+n-4)(n-1)^2\lo^3\mo^3 -6(n+1)(n-1)(\lo\ld\mo^3+\lo^3\mo\md)+ \\
 +16(\lt\mo^3 &+\lo^3\mt+6(n^2+3)\lo\ld\mo\md-24(\lo\ld\, \mt+\lt\mo\md) +8\lt\, \mt\enspace. 
 \end{split}
 \end{equation}
The centered third moment
\begin{equation}
\label{m3cenmom}
\mathbb{E}({\tt CV}_c^3)=\mathbb{E}(({\tt CV}-\mathbb{E}({\tt CV}))^3)=\mathbb{E}({\tt CV}^3)-3\, \mathbb{E}({\tt CV}^2)\, \mathbb{E}({\tt CV})+2\,
\mathbb{E}^3({\tt CV})
\end{equation}
finally reads,  by  (\ref{firstorder}), (\ref{secondorder}) and (\ref{kappahat})
\begin{equation}
\begin{split}
\label{thirdorder}
\frac{\mathbb{E}({\tt CV}_c^3)}{8(n-1)^4\hat{\kappa}} \, =\, & 
4\lo^3\mo^3
-6(\lo\ld\mo^3+\lo^3\mo\md)
+2(\lt\mo^3+\lo^3\mt)
+9\, \lo\ld\mo\md \notag
\\ 
-3(\overline{\bm{\lambda}}\overline{\bm{\lambda}^2}\, \overline{\bm{\mu}^3}+ &  \overline{\bm{\lambda}^3}\overline{\bm{\mu}}\overline{\bm{\mu}^2})+\overline{\bm{\lambda}^3}\, \overline{\bm{\mu}^3}=(\overline{\bm{\lambda}^3}-3\overline{\bm{\lambda}}\overline{\bm{\lambda}^2}+2\overline{\bm{\lambda}}^3)(\overline{\bm{\mu}^3}-3\overline{\bm{\mu}}\overline{\bm{\mu}^2}+2\overline{\bm{\mu}}^3)=  \ltc\: \mtc \notag
\end{split}
\end{equation}
thus proving (\ref{4MainTheo3}). This exact expression for the third moment  seems original, and is considerably simpler than the corresponding expression derived by averaging on the $n!$ object permutations \citep{kazi1995refined}. It depends directly on $n$, but only indirectly on $\mathbf{f}$ through the eigenvalue spectra. Expression (\ref{4MainCoroo3}) for the RV skewness is particularly transparent, and elucidates the cause of the marked positive asymmetry of the RV coefficient, often reported in the literature \citep[see e.g.][]{mielke198434,heo1998permutation,josse2008testing,zhang2009rv}: plainly,  $a({\bm \lambda})>0$ and $ a({\bm \mu})>0$  for typical scree plots.

\

\subsection{The fourth moment}
\label{q4}
Computing $\mathbb{E}({\tt RV}^4)$, or equivalently $\mathbb{E}({\tt CV}^4)$ is clearly untractable with the former pedestrian approach, and a more systematic strategy is needed. The latter is provided by the work around the {\em orthogonal Weingarten functions} \citep[see][]{collins2006integration, collins2009some,matsumoto2012general,collins2013compound,mingo2013real,mingo2017free}.

Consider $\mathcal{P}_{2q}$, the set of all partitions of $\{1,2,\ldots,2q\}$ whose all blocks are of length two, also called {\em pairings}. There are  $(2q-1)!!=(2q-1)(2q-3)\cdots 5\cdot 3$ distinct pairings. For instance, for $q=4$
\begin{displaymath}
\sigma=(13|25|46|78) \qquad\mbox{and}\qquad \tau=(15|26|34|78) 
\end{displaymath}
constitute such pairings. Their {\em join} $\sigma\vee \tau$ (the finest partition coarser than both $\sigma$ and $\tau$) is $\sigma\vee \tau=(123456|78)$. 

In general, the join $\sigma\vee \tau$ of two pairings $\sigma, \tau\in \mathcal{P}_{2q}$ is a partition made of $N(\sigma\vee \tau)$ blocks of even sizes $2l_1, 2l_2,  2l_3,\ldots$, with $l_1\ge l_2\ge l_3\ldots $ and $\sum_{c=1}^{N(\sigma\vee \tau)}l_c=q$. The multi-index $\ell=(l_1,l_2,l_3\ldots)$ constitutes an {\em integer partition} of $q$ (noted $\ell\vdash q$), and defines the {\em type} $\ell(\sigma\vee \tau)$ of $\sigma\vee \tau$.

For $q=4$, five integer partitions or {\em types} are possible, namely 
\begin{displaymath}
\hspace*{-0.3cm}\ell=(1,1,1,1)\equiv(1^4)\qquad  \ell=(2,1,1)\qquad \ell=(2,2)\qquad  \ell=(3,1)\quad\quad \ell=(4)\enspace. 
\end{displaymath}

The orthogonal coefficients (\ref{orthoCoeff}) turn out to express \citep{collins2006integration} as
\begin{equation}
\label{ColSni}
\mathcal{I}^{\bm \omega}_{\mathbf a}=\sum_{\sigma\in  \mathcal{P}_{2q}}
\sum_{\tau\in  \mathcal{P}_{2q}}\: \delta_{\sigma}({\bm \omega})\: \delta_{\tau}({\mathbf a})\:  \mbox{Wg}(\ell(\sigma\vee \tau))
\end{equation}
where (considering now $\sigma$ and $\tau$ as {\em permutations} exchanging the indices belonging to the same block of two), the multi-Kronecker symbols select the pairings $\sigma$  and $\tau$ {\em compatible with the multi-indices},  in the sense
\begin{equation}
\label{multiKronecker}
\delta_{\sigma}({\bm \omega})=\prod_{r=1}^q \delta_{\omega_{\sigma(2r-1)},\omega_{\sigma(2r)}}
\qquad\qquad
\delta_{\tau}({\mathbf a})=\prod_{r=1}^q \delta_{a_{\tau(2r-1)},a_{\tau(2r)}}
\end{equation}
In other words, $\delta_{\sigma}({\bm \omega})=1$ if $\omega_s=\omega_t$ for each pair $(s,t)$ in $\sigma$ (which implies that  all indices in ${\bm \omega}$ must occur an even number of times), and $\delta_{\sigma}({\bm \omega})=0$ otherwise. 

The quantities $\mbox{Wg}(\ell(\sigma\vee \tau))$ appearing in (\ref{ColSni}) are the {\em orthogonal Weingarten functions}, and depend upon the dimension $d=n-1$ as well. They have been computed up to order $q=6$ \citep{collins2009some}. For $q=4$ : 
\begin{align}
\label{WeinValues}
 \mbox{Wg}(1^4)=   &  \:   \phi \:  (n-3)(n+2)(n^2+4n-4) \notag \\
 \mbox{Wg}(2,1,1)=   & \:    \phi \:  (-n^3-3n^2+6n+4)\notag \\
 \mbox{Wg}(2,2)= &  \:  \phi \:  (n^2+3n+14) \\
\mbox{Wg}(3,1)= & \:  \phi \:  (n-1)(2 n+6)\notag \\
\mbox{Wg}(4)=&  \:  - \phi \:  (5n+1)\notag
\end{align}
\begin{equation}
\label{CoeffPhi}
\mbox{where}\qquad\phi= \frac{1}{(n-4)(n-3)(n-2)(n-1)n(n+1)(n+3)(n+5)}\enspace. 
\end{equation}
Substituting (\ref{ColSni}) in (\ref{heavy}) yields 
\begin{align}
\label{MyFormulaforCVq}
\mathbb{E}({\tt CV}^q) = &  \sum_{\sigma\in\mathcal{P}_{2q}}
\sum_{\tau\in \mathcal{P}_{2q}} \mbox{Wg}(\ell(\sigma\vee \tau)) \sum_{{\bm \alpha}\in A_q}\delta_{\sigma}({\bm \alpha}{\bm \alpha})\:  \lambda_{\bm \alpha}\sum_{{\bm \beta}\in A_q}\delta_{\tau}({\bm \beta}{\bm \beta})\: 
 \mu_{\bm \beta}
 \notag  \\
   = &    \sum_{\sigma\in\mathcal{P}_{2q}}\sum_{\tau\in \mathcal{P}_{2q}} \mbox{Wg}(\ell(\sigma\vee \tau))\: \mbox{\rm Tr}_\sigma(\mathbf{K}_X)\: \mbox{\rm Tr}_\tau(\mathbf{K}_Y)\\
   = &     \sum_{\sigma\in\mathcal{P}_{2q}}\sum_{\tau\in \mathcal{P}_{2q}} {\scriptstyle (n-1)^{N(\sigma\vee\sigma_0)+N(\tau\vee\sigma_0)}}\: \:  \mbox{Wg}(\ell(\sigma\vee \tau)) \prod_{c=1}^{N(\sigma\vee\sigma_0)}\overline{{\bm \lambda}^{l_c}}\: \prod_{\tilde{c}=1}^{N(\tau\vee\sigma_0)}\overline{{\bm \mu}^{l_{\tilde{c}}}} \notag
 \end{align}
where the following lemmas and definitions have been used : 
\begin{lemma}[another variant of joint orthogonality, proved in the Appendix]
\label{doublesum}
\begin{displaymath}
\sum_{{\mathbf i}\in [n]^{2q}}v_{{\mathbf i}{\bm \omega}}
\sum_{{\mathbf a}\in A_{2q}}\delta_{\sigma}({\mathbf a})\, u_{{\mathbf i}{\mathbf a}}=\delta_{\sigma}({\bm \omega})
\end{displaymath}
\end{lemma}

\begin{lemma}\label{Ksigma}
Consider the {\em reference pairing} $\sigma_0=(12|34|\ldots|2q-1,q)\in\mathcal{P}_{2q}$, and consider the type  $\ell(\sigma\vee \sigma_0)$, also called {\em coset-type} of $\sigma$ \citep[see e.g.][]{matsumoto2012general, collins2013compound}. Define
\begin{equation}
\label{CosetTrace}
\mbox{\rm Tr}_\sigma(\mathbf{K})=\prod_{c=1}^{N(\sigma\vee\sigma_0)}\mbox{\rm Tr}(\mathbf{K}^{l_c})
\qquad\qquad
\mbox{\rm tr}_\sigma(\mathbf{K})=\prod_{c=1}^{N(\sigma\vee\sigma_0)}\mbox{\rm tr}(\mathbf{K}^{l_c})
\end{equation}
Then $\sum_{{\bm \alpha}\in A_q}\delta_{\sigma}({\bm \alpha}{\bm \alpha})\:  \lambda_{\bm \alpha}=\mbox{\rm Tr}_\sigma(\mathbf{K}_X)$, which also reads 
\begin{equation}
\label{ }
\mbox{\rm Tr}_\sigma(\mathbf{K}_X)=(n-1)^{N(\sigma\vee\sigma_0)}\; \mbox{\rm tr}_\sigma(\mathbf{K}_X)
=(n-1)^{N(\sigma\vee\sigma_0)}\; \prod_{c=1}^{N(\sigma\vee\sigma_0)}\overline{{\bm \lambda}^{l_c}}\enspace.
\end{equation}
\end{lemma}
\begin{proof}
Consider  ${\bm \omega}={\bm \alpha}{\bm \alpha}\in A_{2q}$. By construction, $\omega_{(2r-1)}=\omega_{2r}$, that is $\delta_{\sigma_0}({\bm \omega})=1$. On the other hand, the term $\delta_{\sigma}({\bm \alpha}{\bm \alpha})$ imposes $\omega_{\sigma(2r-1)}=\omega_{\sigma(2r)}$. Hence all indices of $\bm \omega$  in the blocks of $\sigma\vee\sigma_0$ (of sizes 
$2l_1,2l_2,2l_3\ldots,2l_{N(\sigma\vee\sigma_0)}$) are identical, that is the sum on ${\bm \alpha}$ involves $N(\sigma\vee\sigma_0)$ unconstrained indices respectively repeated  exactly  $(l_1,l_2,l_3\ldots,l_{N(\sigma\vee\sigma_0)})$ times.
\end{proof}

Transforming    expression (\ref{MyFormulaforCVq}) into an effective formula requires to determine, among the $((2q-1)!!)^2$  pairings $(\sigma,\tau)$ entering into the sum, how many are jointly of type $\ell(\sigma\vee\tau)$, $\ell(\sigma\vee\sigma_0)$ and $\ell(\tau\vee\sigma_0)$. 

For $q=4$, Table \ref{FiveTables_q4} gives the distribution of joint counts of the $105^2=11025$ pairings $(\sigma,\tau)$, among the $5^3=125$ possible trivariate types. Those counts have,  for lack of foreseeable analytical approach, 
been mechanically computed with  the help of the {\tt R} package {\tt igraph} \citep{csardi2006igraph}, by the functions {\tt union()} (determining the join of two pairings coded as binary graphs) and {\tt components()} (determining the join type).

\begin{table}
  \caption{Trivariate type counts for $q=4$: each table refers to the type of $\sigma\vee\tau$, namely $(1^4)$, $(2,1,1)$, $(2,2)$, $(3,1)$, $(4)$ from left to right and  top to bottom. Rows refer to the type of $\sigma\vee\sigma_0$,  and columns to the type of $\tau\vee\sigma_0$.}
  \label{FiveTables_q4}
  \begin{center}
\includegraphics[width=0.32\textwidth]{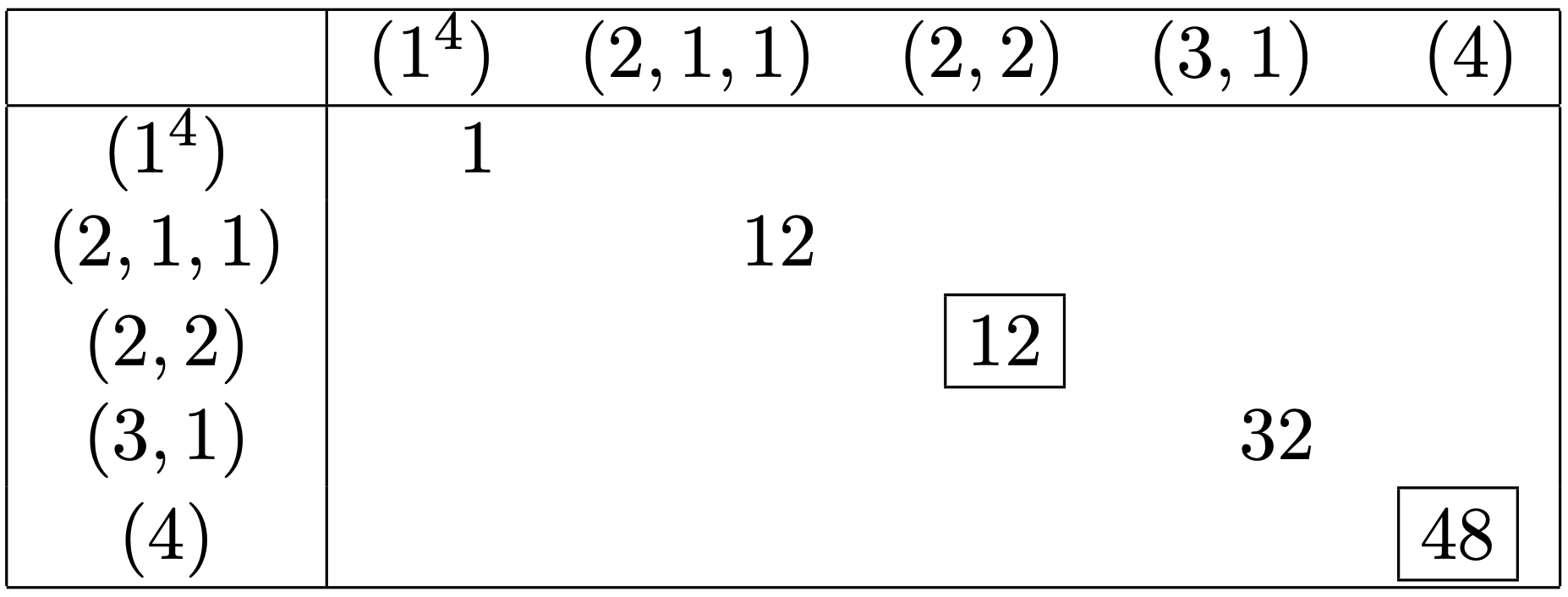}  
\hspace{0cm}
\includegraphics[width=0.32\textwidth]{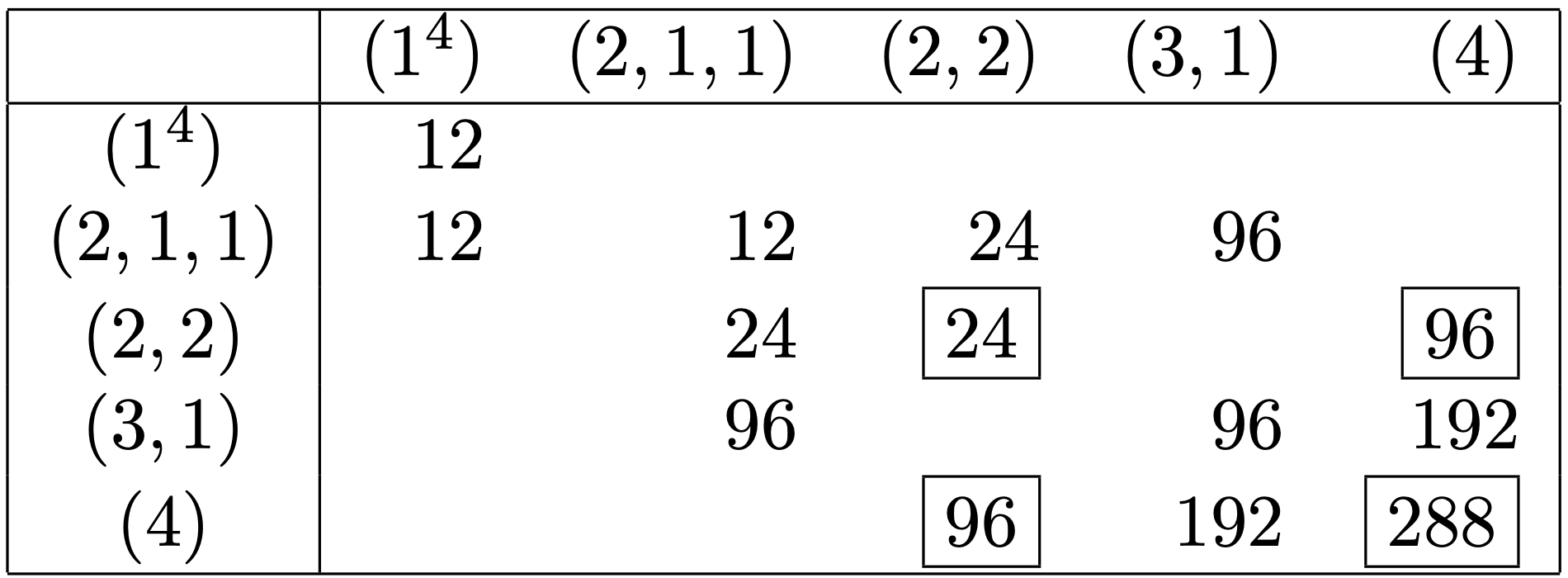}
\hspace{0cm}
\includegraphics[width=0.32\textwidth]{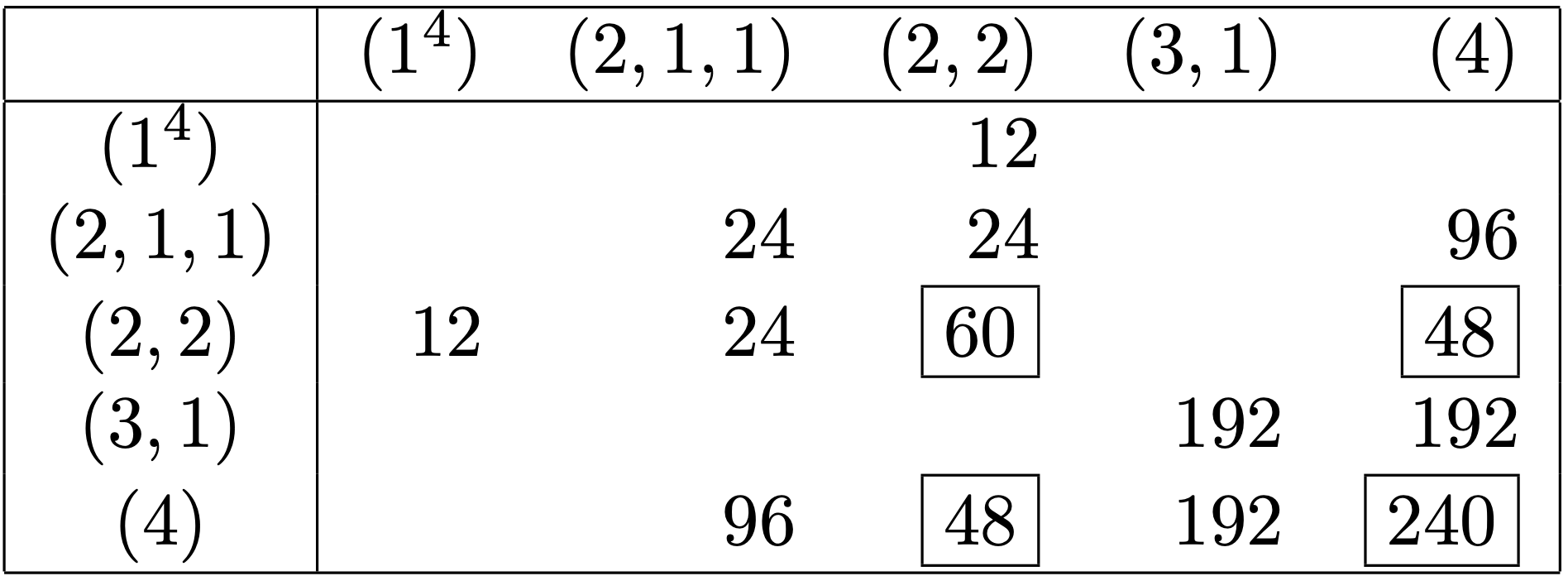}  
\end{center}

\begin{center}
\includegraphics[width=0.32\textwidth]{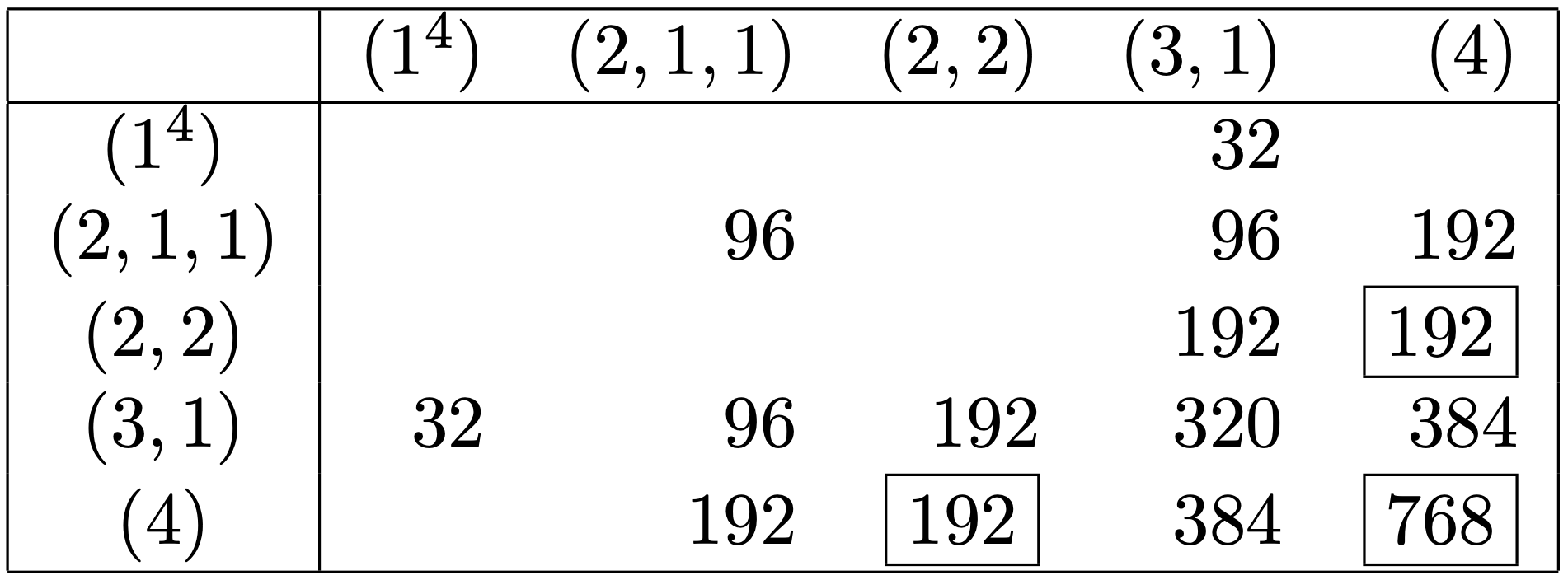}  
\hspace{0cm}
\includegraphics[width=0.32\textwidth]{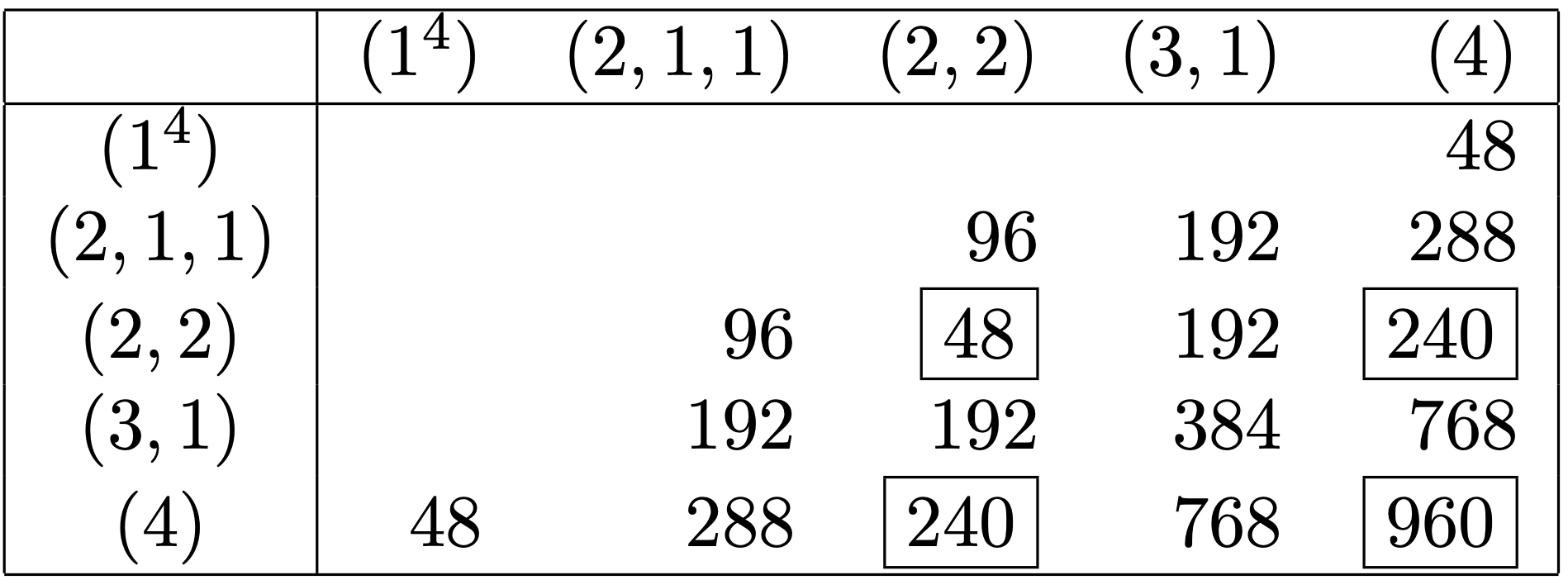}
\end{center}
\end{table}

Working with {\em centered quantities} notably simplifies the computations: 
\begin{lemma}
\label{centered}
\begin{displaymath}
{\tt CV}_c= {\tt CV}-\mathbb{E}({\tt CV})
=\sum_{\alpha=1}^{n-1}\sum_{\beta=1}^{n-1} \lambda^c_{\alpha}\mu^c_{\beta} P_{\alpha\beta}\enspace. 
\end{displaymath}
\end{lemma}
\begin{proof}
By (\ref{4MainTheo1}), (\ref{CVdef1}) and $P_{\alpha\bullet}=P_{\bullet\beta}=1$, 
\begin{align*}
\hspace*{-1cm}&\sum_{\alpha=1}^{n-1}\sum_{\beta=1}^{n-1} \lambda^c_{\alpha}\mu^c_{\beta} P_{\alpha\beta}=
\sum_{\alpha\beta}(\lambda_\alpha-\lo)(\mu_\beta-\mo)P_{\alpha\beta}   = 
\sum_{\alpha\beta} \lambda_\alpha \mu_\beta P_{\alpha\beta}-(n-1)\lo\mo \notag\\
    &   -(n-1)\lo\mo +(n-1)\lo\mo= \sum_{\alpha\beta} \lambda_\alpha \mu_\beta P_{\alpha\beta}-(n-1)\lo\, \mo= {\tt CV}-\mathbb{E}({\tt CV})={\tt CV}_c\enspace. 
\end{align*}$\mbox{}$
\end{proof}
Consequently, (\ref{MyFormulaforCVq}) entails
\begin{equation}
\label{MyFormCentered}
\mathbb{E}({\tt CV}_c^q) =  \sum_{\sigma\in\mathcal{P}_{2q}}\sum_{\tau\in \mathcal{P}_{2q}} {\scriptstyle (n-1)^{N(\sigma\vee\sigma_0)+N(\tau\vee\sigma_0)}}\: \:  \mbox{Wg}(\ell(\sigma\vee \tau)) \prod_{c=1}^{N(\sigma\vee\sigma_0)}\overline{{\bm \lambda}_c^{l_c}}\: \prod_{\tilde{c}=1}^{N(\tau\vee\sigma_0)}\overline{{\bm \mu}_c^{l_{\tilde{c}}}}
\end{equation}
in which, for $q=4$,  the contributions of $\sigma$ and $\tau$ coset  types $(1^4)$, $(2,1,1)$ and $(3,1)$, associated to  $\overline{{\bm \lambda}_c}=0$ or $\overline{{\bm \mu}_c}=0$, are zeroed:  only $(2,2)$ and $(4)$ survive, with contributions indicated by the boxed counts in Table~\ref{FiveTables_q4}. Explicitly, the coefficient of $ \ldc\: \ldc\:  \mdc\:  \mdc$ in  (\ref{MyFormCentered}) is
\begin{align*}
    &   (n-1)^4 [12\, \mbox{Wg}(1^4)+24\,  \mbox{Wg}(2,1,1)+60\,   \mbox{Wg}(2,2)+48\, \mbox{Wg}(4) ] \\
= \:  & 12\, \phi\, (n-1)^4\, (n^4+n^3-15 n^2-13 n+98) \enspace, 
\end{align*}
the coefficient of $\lqc\: \mqc$  is
\begin{align*}
    &   (n-1)^2 [48\, \mbox{Wg}(1^4)+288\,  \mbox{Wg}(2,1,1)+240\,   \mbox{Wg}(2,2)+768 \, \mbox{Wg}(3,1)+960\, \mbox{Wg}(4) ] \\
=  \:   &  48\, \phi\, (n-1)^4\, (n^2-n+2)  \enspace, 
\end{align*}
and the coefficient of $ \lqc\:  \mdc\: \mdc$ and $\ldc\: \ldc\:  \mqc$ is
\begin{align*}
    &   (n-1)^3 [96\,  \mbox{Wg}(2,1,1)+48\,   \mbox{Wg}(2,2)+192 \, \mbox{Wg}(3,1)+240\, \mbox{Wg}(4) ] \\
 =   \:   &  -48\, \phi\, (n-1)^4\, (2n^2-n-7)   \enspace. 
\end{align*}
The final expressions in the above follow from (\ref{WeinValues})  and, together with (\ref{CoeffPhi}), prove  (\ref{4MainTheo4}). They have been further checked with the software {\em Mathematica}. Expression (\ref{4MainTheo4}) for the fourth moment is  relatively simple, but it lacks elegance  and  direct interpretation.

\subsection{The third moment, revisited}
\label{q3rev}
Applying the steps of previous section  for $q=3$  reveals that the contribution of coset-types $(1,1,1)$ and $(2,1)$ for $\sigma$ or $\tau$ is zero by consequence of centration. Hence, only the coset-types $(3)$ contribute  to (\ref{MyFormCentered}), which is therefore simply {\em proportional} to $\ltc\, \mtc$.  The conciseness of expressions (\ref{4MainTheo3}) and (\ref{4MainCoroo3}) is thus elucidated. The proportionality coefficient is determined by the boxed components of Table \ref{ThreeTables_q3} as 
\begin{equation}
\label{coeffQ3}
(n-1)^2 [8\, \mbox{Wg}(1^3)+24\,  \mbox{Wg}(2,1)+32\, \mbox{Wg}(3)] 
= \frac{8(n-1)^3}{(n-3)(n-2)(n+1)(n+3)} 
\end{equation}
which is exactly expression (\ref{4MainTheo3}), obtained much more indirectly in section \ref{q3}. The values of the Weingarten coefficients in   (\ref{coeffQ3}) were 
obtained from \citet{collins2006integration}. They read with the present notations ($d=n-1$ and  (\ref{kappahat}))  as
\begin{displaymath}
\mbox{Wg}(1^3)=    \hat{\kappa} \: (n^2+n-4) \qquad\qquad \mbox{Wg}(2,1)=     -\hat{\kappa} \:  (n+1)\qquad\qquad \mbox{Wg}(3)=2\hat{\kappa}\enspace. 
\end{displaymath}
Those coefficients coincide, in order, with the values $U$, $V$ and $W$ defined in (\ref{seize}) and determined in (\ref{seizesolPRE}), as they must in view of (\ref{ColSni}). 

In conclusion, the pedestrian approach of sections \ref{Computing low-order}, \ref{q12} and \ref{q3} exactly matches the systematic approach of sections \ref{q4} and \ref{q3rev}: a circumstance both expected and relieving, apt to boost confidence in the soundness of the invariant orthogonal integration approach.

\begin{table}
  \caption{Trivariate type counts for $q=3$: each table refers to the type of $\sigma\vee\tau$, namely $(1^3)$, $(2,1)$, $(3)$ from left to right. Rows refer to the type of $\sigma\vee\sigma_0$,  and columns to the type of $\tau\vee\sigma_0$.}
  \label{ThreeTables_q3}
  \begin{center}
\includegraphics[width=0.32\textwidth]{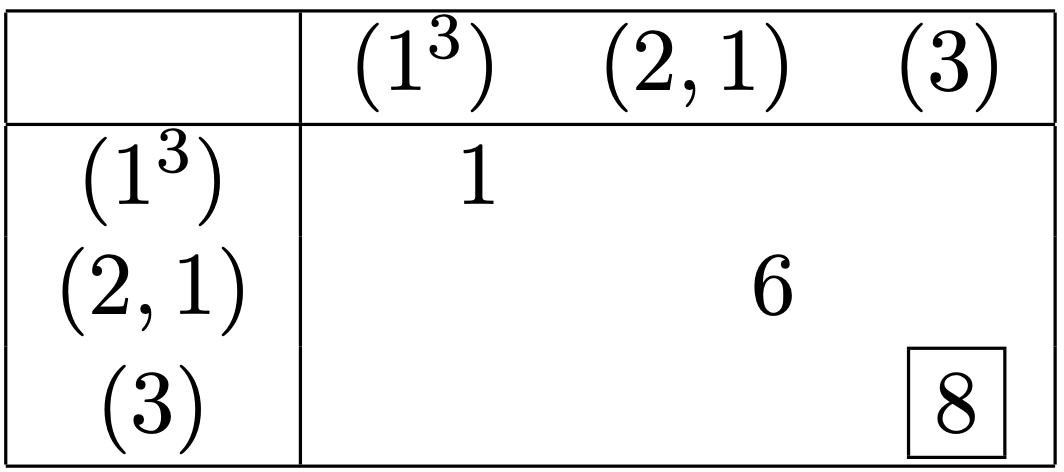}  
\hspace{0cm}
\includegraphics[width=0.32\textwidth]{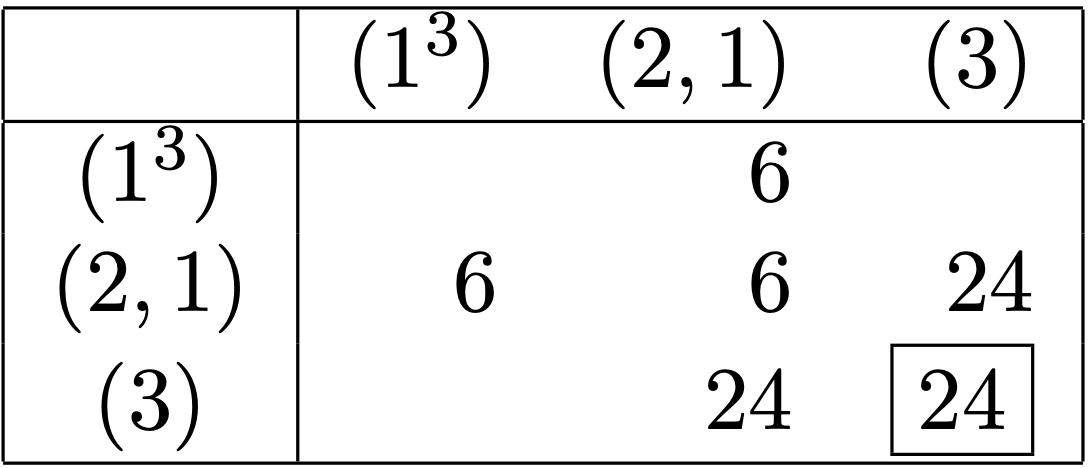}
\hspace{0cm}
\includegraphics[width=0.32\textwidth]{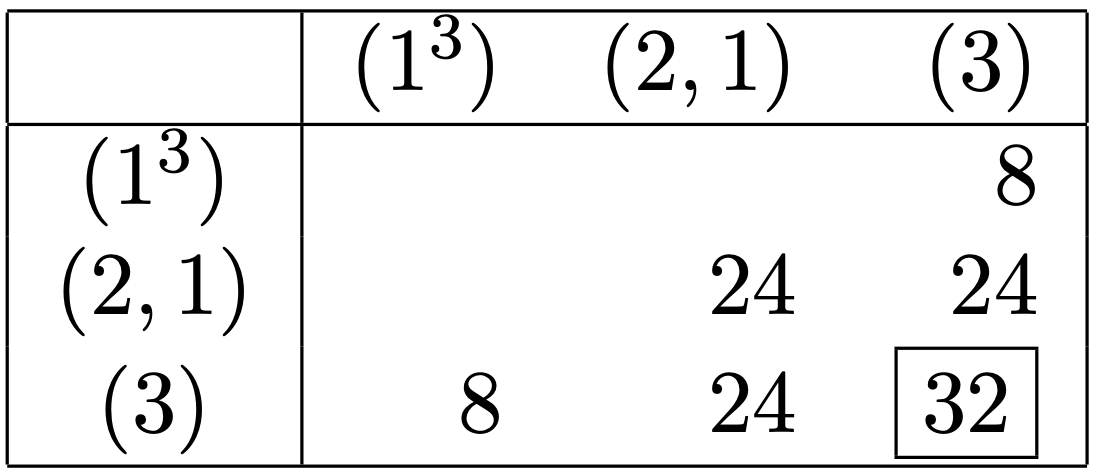}  
\end{center}
\end{table}

\section{Discussion and conclusion}
The weighted RV coefficient measures the similarity between two weighted Euclidean configurations, and this contribution proposes exact expressions for the first four moments of the RV. Considering  weighted objects extends the traditional uniform framework. It also provides  precious guidance for separating the trivial and non-trivial eigenspaces resulting form the spectral decomposition of the standard kernels occurring in the weighted multidimensional scaling of both configurations.

Our approach, invariant orthogonal integration, is nonparametric, and consists in averaging the relative orientation of both configurations by performing Haar integration 
 on orthogonal matrices $\mathbf{T}\in\mathbb{O}_{n-1}$ acting in the non-trivial eigenspace only. The resulting expressions are simpler and easier to interpret than their traditional counterparts obtained by averaging on permutation matrices $\mathbf{S}$ between $n$ objects.  In view of  $\mathbf{S}\mathbf{S}^\top=\mathbf{I}_n$,  permutations also do constitute rotations, but in  $\mathbb{O}_n$, and their undiscriminate use is furthermore questionable in the weighted setting. Comparing the present approach to the parametric approach, postulating multivariate normal distribution for the object features, is left open for future investigations. 
 
 Also, our approach is object-oriented, as in traditional Data Analysis and Machine Learning, rather than variable-oriented as in Mathematical Statistics. Its use requires to dispose of squared Euclidean dissimilarities between objets, possibly weighted, and some of its numerous applications (including spatial autocorrelation and network clustering) will be illustrated in forthcoming publications. This contribution underlines in particular the key role played by the {\em standard kernel}, central to weighted multidimensional scaling, and whose  spectrum governs the values of the RV moments. Correlatively, it appears that the humble {\em scree plot} should deserve more consideration, beyond its limited role in selecting the number of spectral dimensions: mentioning and interpreting its effective dimensionality, skewness and fourth spectral moment could arguably become systematic in practice.
 
 Computing the fourth RV moment did require to recourse to the Weingarten calculus, whose apparatus, arguably demanding for the neophyte, turned out decisive for the pursuit of our objective. One may reasonably hope that future developments along that line will enrich the present results,  replacing in particular the mechanical computation of Tables  \ref{FiveTables_q4} and \ref{ThreeTables_q3} by true mathematical arguments. However, determining the analytical, exact null distribution of ${\tt RV}$, may reveal itself out of reach:  as a matter of fact, the moment generating function (\ref{MomGneFun}) is an orthogonal analog of the celebrated Harish-Chandra {\em trace integral for the unitary group}, whose analytical expression has been determined ever since the fifties \citep{harish1957differential} (see also e.g. \citet{tao2013harish} and  \citet{mcswiggen2021harish}). Yet, discovering a corresponding expression for the orthogonal case, precisely, has not been achieved so far.


\begin{appendix}
\section*{Proofs}\label{appn} 

\begin{proof}[Proof of Lemma \ref{Mainlemma1}]
Let $\alpha\neq\gamma$ and consider the matrix $\tilde{\bm{T}}$ with components 
\begin{equation*}
\label{Givens}
\tilde{t}_{a\alpha}=\cos\xi\,  t_{a\alpha}-\sin\xi\,  t_{a\gamma}
\qquad
\quad
\tilde{t}_{a\gamma}=\sin\xi\,  t_{a\alpha}+\cos\xi\,  t_{a\gamma}
\qquad
\quad
\tilde{t}_{a\beta}=t_{a\beta}\quad\mbox{for } \beta\neq\alpha, \gamma
\end{equation*}
for any $a$, where $\xi$ is an arbitrary, fixed angle. Then $\tilde{\bm{T}}$ is an orthogonal matrix, as likely as $\bm{T}$, that is $d\mu(\tilde{\bm{T}})=d\mu(\bm{T})$. To ease the notations, take ${\bm \varepsilon}$ and $ {\bm e}$ in Lemmas \ref{Mainlemma1} and \ref{Mainlemma2}, playing no active role in what follows, as empty. Then
\begin{align*}
\label{}
\mathcal{I}_{abcd}^{\alpha\alpha\gamma\gamma} &=\mathbb{E}([\cos\xi\, t_{a\alpha}-\sin\xi\, t_{a\gamma}][\cos\xi\, t_{b\alpha}-\sin\xi\, t_{b\gamma}][\sin\xi\, t_{c\alpha}+\cos\xi\, t_{c\gamma}][\sin\xi\, t_{d\alpha}+\cos\xi\, t_{d\gamma}])      \\
    & = 2\cos^2\xi\sin^2\xi \mathcal{I}_{abcd}^{\alpha\alpha\alpha\alpha} + (\cos^4\xi+\sin^4\xi) \mathcal{I}_{abcd}^{\alpha\alpha\gamma\gamma}-2\cos^2\xi\sin^2\xi \mathcal{I}_{abcd}^{\alpha\gamma\alpha\gamma}-2\cos^2\xi\sin^2\xi \mathcal{I}_{abcd}^{\alpha\gamma\gamma\alpha}
\end{align*}
Multiplying the l.h.s. by $\cos^4\xi+\sin^4\xi+2\cos^2\xi\sin^2\xi =1$ and simplifying yields 
\begin{equation*}
\label{ }
 \mathcal{I}_{abcd}^{\alpha\alpha\alpha\alpha}=\mathcal{I}_{abcd}^{\alpha\alpha\gamma\gamma}+\mathcal{I}_{abcd}^{\alpha\gamma\alpha\gamma}+ \mathcal{I}_{abcd}^{\alpha\gamma\gamma\alpha}\enspace. 
\end{equation*}
\end{proof}

\begin{proof}[Proof of Lemma \ref{ThirdOrderP}]
Lemmas (\ref{Mainlemma1}) and (\ref{Mainlemma2}) entail the following relations between orthogonal coefficients (\ref{seize})
\begin{align*}
S&\stackrel{(\ref{lemme1})}{=}U+2V   &   T&\stackrel{(\ref{lemme1})}{=}V+2W &  M&\stackrel{(\ref{lemme2})}{=}3N &  P&\stackrel{(\ref{lemme2})}{=}3S   \notag\\
R&\stackrel{(\ref{lemme2})}{=}3T  &   E&\stackrel{(\ref{lemme3})}{=}(n-2)M+L &  E&\stackrel{(\ref{lemme3})}{=}(n-2)P+M &  F&\stackrel{(\ref{lemme3})}{=}(n-3)N+2M  \notag\\
F&\stackrel{(\ref{lemme3})}{=}(n-3)S+2Q  &  F&\stackrel{(\ref{lemme3})}{=}(n-2)Q+M  & G&\stackrel{(\ref{lemme3})}{=}(n-3)U+2S  &  H&\stackrel{(\ref{lemme3})}{=}(n-3)V+2T   \notag  \\
0&\stackrel{(\ref{lemme3})}{=}(n-2)R+M  &  0&\stackrel{(\ref{lemme3})}{=}(n-2)T+N   & 0&\stackrel{(\ref{lemme3})}{=}(n-3)W+2T &   &
\notag
\end{align*}
with solution (recall that $E,F,G,H$ in (\ref{solkappa}) are already known) 
\begin{equation}
\label{seizesolPRE}
\begin{split}    
L&=\frac{15(n-2)}{n+3}\,\kappa         \qquad   M =\frac{3(n-2)}{n+3}\,\kappa     \qquad   N =\frac{n-2}{n+3}\,\kappa     \qquad   P =\frac{3(n+2)}{n+3}\,\kappa 
\\
Q&=\frac{n}{n+3}\,\kappa   \qquad      R =\frac{-3}{n+3}\,\kappa   \qquad   S =\frac{n+2}{n+3}\,\kappa \qquad    T =\frac{-1}{n+3}\,\kappa 
 \\
 U&=\frac{n^2+n-4}{(n-3)(n+3)}\,\kappa  \qquad       V =\frac{-(n+1)}{(n-3)(n+3)}\,\kappa     \qquad      W =\frac{2}{(n-3)(n+3)}\,\kappa \enspace . \end{split}
\end{equation}

Consider first $\alpha=\gamma=\varepsilon$, and assume the sub-indices  of the orthogonal coefficients to be matched into three distinct pairs. There are
$5\times 3=15$ such pairings, namely 
\begin{equation}
\label{age0}
\begin{split}
\mathcal{I}_{abcdef}^{\alpha\alpha\alpha\alpha\alpha\alpha} & = N
\{{\scriptstyle\delta_{ab}\delta_{cd}\delta_{ef}
+\delta_{ac}\delta_{bd}\delta_{ef}
+\delta_{ad}\delta_{bc}\delta_{ef}
+\delta_{ae}\delta_{cd}\delta_{bf}
+\delta_{af}\delta_{cd}\delta_{be}
+\delta_{ab}\delta_{ce}\delta_{df}
+\delta_{ab}\delta_{cf}\delta_{de}
}
\\
+&{\scriptstyle\delta_{ac}\delta_{be}\delta_{df}
+\delta_{ac}\delta_{bf}\delta_{de}
+\delta_{ad}\delta_{ce}\delta_{bf}
+\delta_{ad}\delta_{cf}\delta_{be}
+\delta_{ae}\delta_{bc}\delta_{df}
+\delta_{ae}\delta_{bd}\delta_{cf}
+\delta_{af}\delta_{bc}\delta_{de}
+\delta_{af}\delta_{bd}\delta_{ce}}\} \enspace. 
\end{split}
\end{equation}
In (\ref{age0}), the first term preserves the three pairs in the {\em reference partition} $(ab|cd|ef)$, the next six terms preserve one pair only, and the eight remaining  terms mix all pairs. It turns out that (\ref{age0}) also holds for coinciding pairs in view of $M=3N$ and $L=5M$. By joint orthogonality, the sum in (\ref{CiJ}) reads
\begin{equation}
\label{triple1}
\mathbb{E}(P_{\alpha\beta}P_{\alpha\delta}P_{\alpha\zeta})
=(1+2\delta_{\beta\delta}+2\delta_{\beta\zeta}+2\delta_{\delta\zeta}+8\delta_{\beta\delta}\delta_{\beta\zeta}\delta_{\delta\zeta})N\enspace .
\end{equation}

Consider now $\alpha=\gamma\neq\varepsilon$. Distinguishing between cases preserving or not the pair $(ef)$  yields
\begin{equation}
\label{agNe0}
\begin{split}
\mathcal{I}_{abcdef}^{\alpha\alpha\alpha\alpha\varepsilon\varepsilon} & = S
\{{\scriptstyle
\delta_{ab}\delta_{cd}\delta_{ef}
+\delta_{ac}\delta_{bd}\delta_{ef}
+\delta_{ad}\delta_{bc}\delta_{ef}
}\} 
+T\{{\scriptstyle\delta_{ae}\delta_{cd}\delta_{bf}
+\delta_{af}\delta_{cd}\delta_{be}
+\delta_{ab}\delta_{ce}\delta_{df}
+\delta_{ab}\delta_{cf}\delta_{de}
}
\\
&
{\scriptstyle
+\delta_{ac}\delta_{be}\delta_{df}
+\delta_{ac}\delta_{bf}\delta_{de}
+\delta_{ad}\delta_{ce}\delta_{bf}
+\delta_{ad}\delta_{cf}\delta_{be}
+\delta_{ae}\delta_{bc}\delta_{df}
+\delta_{ae}\delta_{bd}\delta_{cf}
+\delta_{af}\delta_{bc}\delta_{de}
+\delta_{af}\delta_{bd}\delta_{ce}} \} 
\end{split}
\end{equation}
which also holds for three preserved pairs since $3S+12T=M$. By joint orthogonality, (\ref{CiJ}) reads
\begin{align}
\mbox{for $\alpha=\gamma\neq\varepsilon$}\qquad  \mathbb{E}(P_{\alpha\beta}P_{\gamma\delta}P_{\varepsilon\zeta})
&=(1+2\delta_{\beta\delta})S+(2\delta_{\beta\zeta}+2\delta_{\delta\zeta}+8\delta_{\beta\delta}\delta_{\beta\zeta}\delta_{\delta\zeta})T
\label{triple2eps}
\\
\mbox{for $\alpha=\varepsilon\neq\gamma$}\qquad \mathbb{E}(P_{\alpha\beta}P_{\gamma\delta}P_{\varepsilon\zeta})
&=(1+2\delta_{\beta\zeta})S+(2\delta_{\beta\delta}+2\delta_{\delta\zeta}+8\delta_{\beta\delta}\delta_{\beta\zeta}\delta_{\delta\zeta})T
\label{triple2gamma}
\\
\mbox{for $\gamma=\varepsilon\neq\alpha$}\qquad  \mathbb{E}(P_{\alpha\beta}P_{\gamma\delta}P_{\varepsilon\zeta})
&=(1+2\delta_{\delta\zeta})S+(2\delta_{\beta\delta}+2\delta_{\beta\zeta}+8\delta_{\beta\delta}\delta_{\beta\zeta}\delta_{\delta\zeta})T\enspace.
\label{triple2alpha}
\end{align}
In the remaining case $\alpha\neq\gamma\neq\varepsilon$, the same reasoning yield 
\begin{equation}
\label{aNgNeNEW}
\begin{split}\mathcal{I}_{abcdef}^{\alpha\alpha\gamma\gamma\varepsilon\varepsilon}& = U
\{{\scriptstyle
\delta_{ab}\delta_{cd}\delta_{ef}
}\}
+V\{{\scriptstyle
\delta_{ac}\delta_{bd}\delta_{ef}
+\delta_{ad}\delta_{bc}\delta_{ef}
+\delta_{ae}\delta_{cd}\delta_{bf}
+\delta_{af}\delta_{cd}\delta_{be}
+\delta_{ab}\delta_{ce}\delta_{df}
+\delta_{ab}\delta_{cf}\delta_{de}
}\}\\
+W&\{{\scriptstyle 
\delta_{ac}\delta_{be}\delta_{df}
+\delta_{ac}\delta_{bf}\delta_{de}
+\delta_{ad}\delta_{ce}\delta_{bf}
+\delta_{ad}\delta_{cf}\delta_{be}
+\delta_{ae}\delta_{bc}\delta_{df}
+\delta_{ae}\delta_{bd}\delta_{cf}
+\delta_{af}\delta_{bc}\delta_{de}
+\delta_{af}\delta_{bd}\delta_{ce}
}\} 
\end{split}
\end{equation}
also valid for three preserved pairs since $U+6V+8W=N$, and finally
\begin{equation}
\label{triple3}
\mbox{for $\alpha\neq\gamma\neq\varepsilon$}\qquad \mathbb{E}(P_{\alpha\beta}P_{\gamma\delta}P_{\varepsilon\zeta})
=U+(2\delta_{\beta\delta}+2\delta_{\beta\zeta}+2\delta_{\delta\zeta})V+8\delta_{\beta\delta}\delta_{\beta\zeta}\delta_{\delta\zeta}W\enspace.
\end{equation}
To ease notations, use  definitions (\ref{NewDiracs}), multiply  both sides of (\ref{triple1}) by $\varphi$, of (\ref{triple2eps}) by $\delta_{\alpha\gamma}(1-\delta_{\alpha\varepsilon})(1-\delta_{\gamma\varepsilon})=\delta_{\alpha\gamma}-\varphi$, of (\ref{triple2gamma}) by $\delta_{\alpha\varepsilon}-\varphi$, of 
  (\ref{triple2alpha}) by  $\delta_{\gamma\varepsilon}-\varphi$,  of 
(\ref{triple3}) by $(1-\delta_{\alpha\gamma})(1-\delta_{\alpha\varepsilon})(1-\delta_{\gamma\varepsilon})=1-\sigma+2\varphi$, and add  the whole to obtain the 
unrestricted expression  (\ref{victoireYes}). 
\end{proof}

\begin{proof}[Proof of Lemma \ref{doublesum}]
\begin{align*}
\sum_{{\mathbf a}\in A_{2q}}\delta_{\sigma}({\mathbf a})\, u_{{\mathbf i}{\mathbf a}}=   &   \sum_{{\mathbf a}\in A_{2q}}\delta_{\sigma}({\mathbf a})\, u_{{\mathbf i}{\mathbf a}}
 \sum_{{\mathbf a}\in A_{2q}}\prod_{r=1}^q \delta_{a_{\sigma(2r-1)},a_{\sigma(2r)}}\, u_{{\mathbf i}{\mathbf a}}
 \\ =& \prod_{r=1}^q \: \sum_{a_{\sigma(2r-1)=1}}^{n-1}u_{a_{\sigma(2r-1)},i_{\sigma(2r-1)}}u_{a_{\sigma(2r-1)},i_{\sigma(2r)}} 
\\  =   & \prod_{r=1}^q [\delta_{i_{\sigma(2r-1)},i_{\sigma(2r)}}-\sqrt{f_{i_{\sigma(2r-1)}}}\sqrt{f_{i_{\sigma(2r)}}}]\enspace. 
\end{align*}
Summing the latter on $\sum_{{\mathbf i}\in [n]^{2q}}v_{{\mathbf i}{\bm \omega}}$ yields, by the joint orthogonality property of  section \ref{invorthin}
$$
\prod_{r=1}^q \: \sum_{i_{\sigma(2r-1)=1}}^{n}v_{i_{\sigma(2r-1)},\omega_{\sigma(2r-1)}}v_{i_{\sigma(2r-1)},\omega_{\sigma(2r)}}
=\prod_{r=1}^q  \delta_{\omega_{\sigma(2r-1)},\omega_{\sigma(2r)}}=\delta_{\sigma}({\bm \omega})\enspace. 
$$
\end{proof}

\end{appendix}




\bibliography{ArXivVersionFirstRvMomentsFB}
\bibliographystyle{apalike}

\end{document}